\documentclass[12pt,a4paper]{article}

\usepackage{amssymb, amsmath, amsthm}

\usepackage[cm]{fullpage}
\usepackage[english]{babel}
\usepackage[pdftex]{graphicx}
\usepackage{booktabs}
\usepackage{xcolor}

\usepackage{hyperref}

\newtheorem{thm}{Theorem}
\newtheorem{lem}{Lemma}
\newtheorem{prop}{Proposition}
\newtheorem{rem}{Remark}

\title{A linear Galerkin numerical method for a quasilinear subdiffusion equation\footnote{This is the accepted version of the manuscript published in \textit{Applied Numerical Mathematics} \textbf{185} (2023), 203-220 with DOI: \url{https://doi.org/10.1016/j.apnum.2022.11.020}}}
\author{\L ukasz P\l ociniczak\thanks{Faculty of Pure and Applied Mathematics, Wroc{\l}aw University of Science and Technology, Wyb. Wyspia{\'n}skiego 27, 50-370 Wroc{\l}aw, Poland, lukasz.plociniczak@pwr.edu.pl}}
\date{}

\begin{document}
\maketitle

\begin{abstract}
We couple the L1 discretization for the Caputo derivative in time with the spectral Galerkin method in space to devise a scheme that solves quasilinear subdiffusion equations. Both the diffusivity and the source are allowed to be nonlinear functions of the solution. We prove the stability and convergence of the method with spectral accuracy in space. The temporal order depends on the regularity of the solution in time. Furthermore, we support our results with numerical simulations that utilize parallelism for spatial discretization. Moreover, as a side result, we find exact asymptotic values of the error constants along with their remainders for discretizations of the Caputo derivative and fractional integrals. These constants are the smallest possible, which improves previously established results from the literature. \\


\end{abstract}

\section{Introduction}
We are interested in numerically solving the following time-fractional nonlinear problem with homogeneous Dirichlet conditions
\begin{equation}
\label{eqn:MainPDE}
\begin{cases}
	\partial^\alpha_t u = \left(D(u) u_x\right)_x + f(x, t, u), & x\in(0,1), \quad t \in (0,T), \quad \alpha\in(0,1), \\
	u(x,0) = \varphi(x), & \\
	u(0,t) = 0, \; u(1,t) = 0, & \\
\end{cases}
\end{equation}
where the Caputo derivative is defined for sufficiently smooth functions as \cite{Li19}
\begin{equation}
\label{eqn:Caputo}
\partial^\alpha_t u(x,t) = \frac{1}{\Gamma(1-\alpha)} \int_0^t (t-s)^{-\alpha} u_t(x,s)ds,
\end{equation}
which can also be written in terms of the fractional integral
\begin{equation}
\label{eqn:FracInt}
I^\alpha_t u(x,t) = \frac{1}{\Gamma(\alpha)} \int_0^t (t-s)^{\alpha-1} u(x,s) ds,
\end{equation}
as $\partial^\alpha_t u = I^{1-\alpha}_t u_t$. In this work, we assume that $\alpha\in(0,1)$ corresponds to the subdiffusive regime of the process. Probabilistically, this type of evolution can arise when a randomly walking particle exhibits waiting times that lower the mean square displacement \cite{Met00,Kla12}. For example, in the study of a porous medium, water percolates the solid matrix. For certain materials, it can become trapped in some of its regions or undergo a chemical reaction that also produces a subdiffusive characteristic of the process \cite{Plo15,Plo14}. All of these situations in which the modeling with a time-fractional derivative is relevant point to some instance of the memory of the process that has an essential influence on the present state \cite{Kla12}. On the other side of the diffusion spectrum is the faster counterpart of this process, superdiffusion. It arises when the particle can jump for substantial distances in a given instant \cite{Met00}. A model for that phenomenon frequently involves operators that are nonlocal in space such as the fractional Laplacian \cite{Plo19, Vaz17}. Consequently, for modeling both waiting times and long jumps, one usually uses operators nonlocal in time and space. Subdiffusion plays an important role in many different fields of science and engineering such as turbulent flow \cite{Hen02}, material science \cite{Mul96}, viscoelasticity \cite{Amb96}, cell biochemistry \cite{Sun17}, biophysics \cite{Kou08}, plasma physics \cite{Del05}, and hydrology \cite{El04,El20}.

There is a growing literature on the existence, uniqueness, and regularity of solutions to the linear case of (\ref{eqn:MainPDE}). We would like to mention only several of the multitude of interesting papers. For example, in the seminal work \cite{Sak11} the authors established the unique existence of a weak solution for the case with the symmetric linear uniformly elliptic operator and $f(x,t,u) = f(x,t)$. Moreover, they have given many estimates regarding the asymptotic behavior of solutions for small and large times, along with regularity estimates. Further results concerning the abstract setting of time-fractional equations in Sobolev spaces can be found in \cite{Gor15}, \cite{All16} or \cite{Dip21}. Some other extensions and generalizations can be found, for example, in \cite{Zha19}. On the other hand, to the best of our knowledge, the study of nonlinear time-fractional diffusion equations is much less mature. In \cite{Dip19} authors studied a very general form of (\ref{eqn:MainPDE}) and provided decay estimates of the solution. Furthermore, bounded weak solutions of a variant of (\ref{eqn:MainPDE}) with a generalized form of the memory kernel have been investigated in \cite{Ver15,Wit21}. In \cite{Liu18} the authors studied the existence and uniqueness of the stochastic version of nonlinear PDE. The existence of a unique strong solution of the quasilinear case was established in \cite{Zac12} using the interior H\"older continuity. The viscosity solution method has been applied to a quasilinear subdiffusion equation in \cite{Top17}. Finally, we mention the time-fractional porous medium equation that has been analyzed in \cite{Plo14, Plo15} where some approximate solutions have been given. Moreover, in \cite{Plo17a} the existence and uniqueness of a self-similar solution have also been proved. 

Various numerical methods have been proposed to solve (\ref{eqn:MainPDE}) in the linear case. While spatial discretization is usually done with finite differences, finite elements, or spectral methods, the Caputo derivative (or its generalizations) is discretized mostly by the L1 scheme, Gr"unwald-Letnikov approximation, or convolution quadrature. The latter utilizes the special structure of the fractional derivative as a convolution operator and is based on discretization in the Laplace transform variable. This method has been introduced to solve integral equations \cite{Lub88,Sch06} and provides a very useful scheme for generating versatile discretizations of nonlocal operators \cite{Cue06,Zen15}. On the other hand, Gr\"unwald-Letnikov discretizationThe Gr"unwald-Letnikov is a generalization of classical finite differences applied to the temporal derivative \cite{Li19}. In \cite{Yus05} an explicit finite difference scheme has been utilized to solve the linear variant of (\ref{eqn:MainPDE}). In \cite{Cui09} the authors used a compact scheme to improve spatial accuracy. Gr\"unwald-Letnikov discretization has also been used to solve subdiffusive Fokker-Planck equation \cite{Liu04}. Finally, there exists a family of numerical methods based on piecewise linear approximation of the differentiated function, the L1 scheme originally proposed in \cite{Old74} to discretize the Riemann-Liouville derivative. One of the first works to utilize this type of approximation was \cite{Lan05} in which the authors used the implicit finite difference for the space derivative. Some other approaches and generalizations with similar spatial discretizations have been conducted, for example, in \cite{Zhu06,Liu07}. The L1 method is also conveniently coupled with Finite Element space discretization. In \cite{Zhu06} this scheme has been used to solve the cable equation while in \cite{Zhu16} authors provided a fully abstract setting in the Sobolev spaces. The important problem for methods with non-smooth data has been solved in \cite{Jin13,Lia18,Jin16}. It arises in many cases, since the solution of the linear subdiffusion equation frequently can exhibit a weak temporal singularity in the derivative (see \cite{Sak11,Sty16} for the estimates and \cite{Jin19} for a review of numerical methods). Similarly to the finite element methods, the spectral schemes have been proposed in order to provide superb spatial accuracy with minimal computational expense. Computing nonlocal operators is usually very demanding on processing power, and hence one usually wants to optimize the simulation as a whole. Spectral methods and their exponential order are one way to deal with this problem \cite{Can07}. For this approach, we refer to \cite{Lin07}. A very interesting idea of using the spectral approximation for \emph{both} space and time has been proposed in \cite{Li09}. Since the PDE has a nonlocality in time, using global temporal basis instead of time-stepping is well justified and advantageous. We also refer to \cite{Mus09} for a discontinuous Galerkin method used to solve equations similar to ours. Further generalizations of the L1 scheme involve Crank-Nicolson methods \cite{Jin18} or central differences \cite{Li10} to improve temporal order \cite{Li14}. Additional interesting results on high order \cite{wang2022second,cen2021second,cen2022time,ou2022mathematical,qiao2020alternating}. 
In contrast to the above, the analysis of numerical methods when (\ref{eqn:MainPDE}) involves a nonlinearity is much less evolved. This is mostly due to the difficulty to resolve the interplay between nonlocality and nonlinearity in the discrete setting. According to the best knowledge of the author, most papers deal with a nonlinear source only. In \cite{Li19c} a linearized Galerkin scheme has been proposed to solve the quasilinear equation with a graded temporal mesh. A very useful tool for dealing with nonlinear equations, namely the discrete fractional Gr\"onwall inequality, was developed in \cite{Li18} and later generalized in \cite{Lia18,Lia19}. Subsequently, this inequality has been used in \cite{Ren20} to devise an efficient two-level linearized Finite Element scheme. In addition, an approach based on the spectral approximation has been applied to the nonlocal in time and space (due to the Riesz derivative) in \cite{Zak20}. Finally, we mention our previous works on (\ref{eqn:MainPDE}) with quasilinear case, i.e. when $D(u)\neq $ const. In \cite{Plo19,Plo14} we have considered the time-fractional porous medium equation with $D(u) = u^m$. Our approach was based on providing a family of numerical methods for the setting of the nonlinear Volterra equation that can be obtained in that case by suitable transformations that yield the Erd\'elyi-Kober fractional operator \cite{Plo17,Plo17b}. Further results can be found in \cite{Okr21} and we also refer to a recent result on the spectral method applied to a nonlinear and nonlocal equation that arises in climatology \cite{Plo21}. A comprehensive survey of numerical methods for fractional differential equations can be found in \cite{Li19a}.

In this paper we devise a linear Galerkin spectral method with L1 temporal scheme to solve (\ref{eqn:MainPDE}). We note, however, that our proofs can be readily adapted to the Finite Element method for spatial discretization (see, for example, \cite{Tho07}). The novelty of this work is considering the quasilinear diffusivity and not only the source. First, in the next section, we start with a discussion of the L1 scheme and as a side-result we obtain a sharp estimate on the error constant for this approximation, which improves the known bounds from the literature. Reasoning is carried out with the use of asymptotic analysis, and we also apply it to the fractional integral. However, this case yields a result of a much different nature. In Section 3 we devise the aforementioned Galerkin method. The linearization is achieved by extrapolating all nonlinearities. In the proofs, we also use the discrete fractional Gr\"onwall lemma from \cite{Li18}. Despite the strong non-linearity of the main PDE, our method retrieves the $2-\alpha$ temporal order for smooth solutions and spectral accuracy in space, which is verified by several numerical experiments programmed in Julia. 

\section{The L1 scheme}
Before we devise a numerical method for solving (\ref{eqn:MainPDE}) we focus on the L1 discretization of the Caputo derivative. Let $y: (0,T) \mapsto \mathbb{R}$ be a function defined on the temporal domain. Moreover, set up the time mesh
\begin{equation}
\label{eqn:TimeMesh}
t_n = n h \leq T,
\end{equation} 
where $h>0$ is the time step. One very common choice of discretization is to apply the rectangle product integration rule to the integral in (\ref{eqn:Caputo}) along with a simple finite difference to arrive at the so-called L1 scheme (see \cite{Li19})
\begin{equation}
\label{eqn:L1Scheme}
\delta^\alpha y(t_n) = \frac{h^{-\alpha}}{\Gamma(2-\alpha)} \sum_{i=0}^{n-1} b_{n-i}(1-\alpha) (y(t_{i+1})-y(t_i)),
\end{equation}
with weights
\begin{equation}
\label{eqn:Weights}
b_j(\beta) = j^\beta - (j-1)^\beta. 
\end{equation}
The above scheme can also be written in the following way that is easier to implement in numerical schemes
\begin{equation}
\label{eqn:L1Scheme2}
\delta^\alpha y(t_n) = \frac{h^{-\alpha}}{\Gamma(2-\alpha)}\left(y(t_n) - b_n y(0) -  \sum_{i=1}^{n-1} \left(b_{n-i}(1-\alpha) - b_{n-i+1}(1-\alpha)\right) y(t_i) \right).
\end{equation}
By $R_n(h)$ denote the remainder of the L1 approximation, that is,
\begin{equation}
\label{eqn:L1Error}
\left|\partial^\alpha y(t_n) - \delta^\alpha y(t_n) \right| =: \epsilon_n(h)
\end{equation}
It is a well-known fact that for twice differentiable functions the above L1 scheme has an order of accuracy equal to $2-\alpha$, i.e. (see \cite{Li19,Li19a})
\begin{equation}
\label{eqn:L1Order}
\epsilon_n(h) \leq C_n h^{2-\alpha} \max_{t\in (0,T)} |y''(t)|, 
\end{equation}
where, to keep the notation compact, we retained the symbol $\partial$ for denoting an ordinary derivative. It is important to note that the L1 method only achieves its optimal order $2-\alpha$ for \emph{twice continuously differentiable} functions. When the regularity is less than that, the order deteriorates to $1$. This is especially evident when applied to solving PDEs, as was previously observed in the literature \cite{Kop19} and we will indicate this in the sequel. 

In the literature, there are several results concerning bounds on the error constant $C_n$ in (\ref{eqn:L1Order}). For instance, in \cite{Lin07} it was found that 
\begin{equation}
C_n \leq \frac{\zeta(1+\alpha)}{\Gamma(2-\alpha)},
\end{equation}
where $\zeta$ is the Riemann zeta function. As can be seen, this estimate can be very large for $\alpha$ close to $0$. Some more accurate bounds were given in \cite{Li19a} (p. 106, Theorem 4.1)
\begin{equation}
\label{eqn:L1SchemeCai}
C_n \leq \frac{1}{\Gamma(2-\alpha)} \left(\frac{1-\alpha}{12} + \frac{2^{2-\alpha}}{2-\alpha} - \left(2^{1-\alpha}+1\right)\right).
\end{equation}
In the results below, we find the optimal sharp estimate on $C_n$.

\begin{thm}\label{thm:L1Scheme}
Let $y\in C^2[0,T]$ and $0<\alpha<1$. Then, for any $t\in (0,T]$ the constant $C_n$ defined in (\ref{eqn:L1Order}) satisfies
\begin{equation}
	\label{eqn:L1SchemeOptimalC}
	C_n \sim -\frac{\zeta(\alpha-1)}{\Gamma(2-\alpha)} - \frac{1}{12\Gamma(1-\alpha)}\frac{1}{n^\alpha} \quad \text{as} \quad n\rightarrow\infty, \quad n h \rightarrow t,
\end{equation}
where $\sim$ denotes the asymptotic equality. 
\end{thm}
\begin{proof}
For any $t\in (0,T]$ we have a sequence $\left\{t_n\right\}$ converging to $t$ as $n\rightarrow\infty$ and $nh\rightarrow t$. It follows that $h\rightarrow 0^+$. Fix $t_n$ and subdivide the interval $[0,t_n]$ into pieces of length $h$.  Then interpolating $y$ with a linear function on each subinterval $[t_{i-1}, t_i]$ we have
\begin{equation}
	y(s) = y(t_{i-1}) + \frac{y(t_i) - y(t_{i-1})}{h}(s-t_{i-1}) + \frac{1}{2} y''(t_{i-1}) (s-t_{i-1})(s-t_i) + O(h^3), \quad s\in [t_{i-1}, t_i],
\end{equation}
which leads to
\begin{equation}
	\label{eqn:L1SchemeDerInterp}
	y'(s) = \frac{y(t_i) - y(t_{i-1})}{h}+ y''(t_{i-1}) \left(s- \frac{t_{i-1} + t_i}{2}\right) + O(h^2),
\end{equation}
as $h\rightarrow 0^+$. Plugging this expression into (\ref{eqn:Caputo}) yields the following
\begin{equation}
	\label{eqn:L1SchemeRemainder}
	\begin{split}
		\partial^\alpha y(t_n) 
		&= \frac{h^{-1}}{\Gamma(1-\alpha)} \sum_{i=1}^n (y(t_i) - y(t_{i-1})) \int_{t_{i-1}}^{t_i} (t_n-s)^{-\alpha} ds \\
		&- \frac{1}{\Gamma(1-\alpha)} \sum_{i=1}^n \int_{t_{i-1}}^{t_i} (t_n-s)^{-\alpha} \left(s- \frac{t_{i-1} + t_i}{2}\right)y''(t_{i-1}) + O(h^2) ds \\
		&=: \delta^\alpha y(t_n) - R_n(h) + o(R_n(h)),
	\end{split}
\end{equation}
where we defined the remainder with $h\rightarrow 0^+$. We claim that $R_n(h) = O(h^{2-\alpha})$, which by a simple integration would mean that the term $o(R_n(h))$ above is $O(h^{3-\alpha})$. First, by a straightforward calculation we have
\begin{equation}
	\begin{split}
		&c_{n-i} := \int_{t_{i-1}}^{t_i} (t_n-s)^{-\alpha} \left(s- \frac{t_{i-1} + t_i}{2}\right) ds \\
		&=h^{2-\alpha} \left[\frac{1}{1-\alpha}\left((n-i+1)^{1-\alpha}-(n-i)^{1-\alpha}\right)\left(n-i+\frac{1}{2}\right) - \frac{1}{2-\alpha}\left((n-i+1)^{2-\alpha}-(n-i)^{2-\alpha}\right)\right]
	\end{split}	 
\end{equation}
To sum up the terms $c_{n-i}$ in the definition of $R_n(h)$ in (\ref{eqn:L1SchemeRemainder}) we would like to use the mean value theorem to take $y''(t_{i-1})$ outside the sum (we can move it outside the integral since it is $s$-independent). However, to invoke it, we have to ascertain that $c_i$ is positive for each $i = 1, ..., n$. This is not obvious since the linear function in the integrand, i.e. $s-(t_{i-1}+t_i)/2$ changes sign. To proceed, we write the coefficients differently
\begin{equation}
	c_{n-i} = (h (n-i))^{2-\alpha} \left[\frac{1}{1-\alpha}\left(\left(1+\frac{1}{n-i}\right)^{1-\alpha}-1\right)\left(1+\frac{1}{2(n-i)}\right)-\frac{1}{2-\alpha}\left(\left(1+\frac{1}{n-i}\right)^{2-\alpha}-1\right)\right].
\end{equation}
Now, introducing an auxiliary function
\begin{equation}
	g(x) := \frac{1}{1-\alpha}\left(\left(1+\frac{1}{x}\right)^{1-\alpha}-1\right)\left(1+\frac{1}{2x}\right)-\frac{1}{2-\alpha}\left(\left(1+\frac{1}{x}\right)^{2-\alpha}-1\right),
\end{equation}
we can use the Taylor series for $x^{-1} \rightarrow 0$ to obtain the estimate
\begin{equation}
	g(x) \geq \frac{\alpha}{12 x^3} - \frac{\alpha(1+\alpha)}{24 x^4} = \frac{\alpha}{12 x^3}\left(1-\frac{1+\alpha}{2x}\right) \geq 0,
\end{equation}
for $x\geq 1$. This shows that $c_{n-i} \geq 0$ for $i = 1, ..., n-1$. The last term can be inspected by hand
\begin{equation}
	c_0 = h^{2-\alpha} \left(\frac{1}{2}\frac{1}{1-\alpha} -\frac{1}{2-\alpha} \right) =  \frac{\alpha h^{2-\alpha}}{2(1-\alpha)(2-\alpha)} > 0.
\end{equation}
Therefore, by the mean value theorem with some $\tau_n\in(0,t_n)$ we can write
\begin{equation}
	\begin{split}
		R_n(h) 
		&= \frac{h^{2-\alpha}}{\Gamma(1-\alpha)} y''(\tau_n) \left[\frac{1}{1-\alpha} \sum_{i=1}^n \left((n-i+1)^{1-\alpha}-(n-i)^{1-\alpha}\right)\left(n-i+\frac{1}{2}\right) \right.\\
		&\left.-\frac{1}{2-\alpha} \sum_{i=1}^n \left((n-i+1)^{2-\alpha}-(n-i)^{2-\alpha}\right) \right] \\
		&=: \frac{h^{2-\alpha}}{\Gamma(1-\alpha)} y''(\tau_n) \left[S_1+S_2\right].
	\end{split}
\end{equation}
We will evaluate the two sums above. First, the second is telescoping.
\begin{eqnarray}
	S_2 = -\frac{n^{2-\alpha}}{2-\alpha}. 
\end{eqnarray}
The first one, in turn, can be simplified with the summation by parts (see, for example, \cite{Knu89}). To this end, define the two sequences $f_i$ and $g_i$ as follows
\begin{equation}
	S_1 = \frac{1}{1-\alpha} \sum_{i=1}^n \left((n-i+1)^{1-\alpha}-(n-i)^{1-\alpha}\right)\left(n-i+\frac{1}{2}\right) =: -\frac{1}{1-\alpha} \sum_{i=1}^n (f_i-f_{i-1}) g_i. 
\end{equation}
Then, the summation by parts formula yields
\begin{equation}
	S_1=-\frac{1}{1-\alpha} \left(g_n f_n - g_1 f_0 -\sum_{i=1}^{n-1} f_i (g_{i+1} - g_i) \right) = \frac{1}{1-\alpha}\left(\left(n-\frac{1}{2}\right) n^{1-\alpha} - \sum_{i=1}^{n-1} i^{1-\alpha}\right).
\end{equation}
Putting the values of $S_1$ and $S_2$ together, we arrive at the following
\begin{equation}
	R_n(h) = \frac{h^{2-\alpha}}{\Gamma(1-\alpha)} y''(\tau_n) \left[-\frac{1}{2-\alpha}n^{2-\alpha} + \frac{n+\frac{1}{2}}{1-\alpha}n^{1-\alpha} - \frac{1}{1-\alpha} \sum_{i=1}^{n} i^{1-\alpha} \right],
\end{equation}
where we have added and subtracted the $n$-th term in the sum above to recognize it as the harmonic number of order $\alpha-1$. The last step is to use its asymptotic form (see \cite{Edw74}, Chapter 6.4). For $n\rightarrow\infty$, we have
\begin{equation}
	\label{eqn:HarmonicNumberAsymptotics}
	\sum_{i=1}^{n} i^{1-\alpha} \sim \zeta(\alpha-1) +\frac{1}{2-\alpha} n^{2-\alpha} + \frac{1}{2} n^{1-\alpha} + \frac{1-\alpha}{12} n^{-\alpha}. 
\end{equation} 
Therefore, as $n\rightarrow\infty$ we have
\begin{equation}
	\begin{split}
		R_n(h)
		&\sim \frac{h^{2-\alpha}}{\Gamma(1-\alpha)} y''(\tau_n) \left[-\frac{1}{2-\alpha}n^{2-\alpha} + \frac{n+\frac{1}{2}}{1-\alpha}n^{1-\alpha} \right.\\ 
		&\left.- \frac{1}{1-\alpha} \left(\zeta(\alpha-1) +\frac{1}{2-\alpha} n^{2-\alpha} + \frac{1}{2} n^{1-\alpha} + \frac{1-\alpha}{12} \frac{1}{n^{\alpha}}\right) \right] \\
		&= -\frac{h^{2-\alpha}}{\Gamma(2-\alpha)} y''(\tau_n) \left[\zeta(\alpha-1) + \frac{1-\alpha}{12} \frac{1}{n^\alpha} \right],
	\end{split}
\end{equation}
which is the desired result. 
\end{proof}
Due to the asymptotic form of the constant $C_n$ as in (\ref{eqn:L1SchemeOptimalC}) its leading order value is the smallest possible. For example, the comparison with the value from \cite{Li19a}, that is, (\ref{eqn:L1SchemeCai}), is presented on the left of Fig. \ref{fig:L1Scheme}. We note that the values are close for all $\alpha\in(0,1)$. This can also be verified numerically by choosing $y(t) = t^2/2$ so that $y''(t) = 1$. For $T=1$ we have $h = n^{-1}$ and plot the value of
\begin{equation}
\label{eqn:L1SchemeRho}
\rho_n := 12 \Gamma(1-\alpha) n^\alpha \left|\frac{\partial^\alpha y(1) - \delta^\alpha y(t_n)}{\frac{1}{n^{2-\alpha}}}+\frac{\zeta(\alpha-1)}{\Gamma(2-\alpha)}\right|
\end{equation}
for different $n$. According to (\ref{eqn:L1SchemeOptimalC}) the above expression should approach $1$ as $n\rightarrow \infty$. The results of our calculations are depicted on the right of Fig. \ref{fig:L1Scheme}. As we can see, all cases converge to $1$ very quickly, even for $\alpha$ close to $1$. Note also the scale of ticks on the y-axis. 

\begin{figure}
\centering
\includegraphics[scale = 0.75]{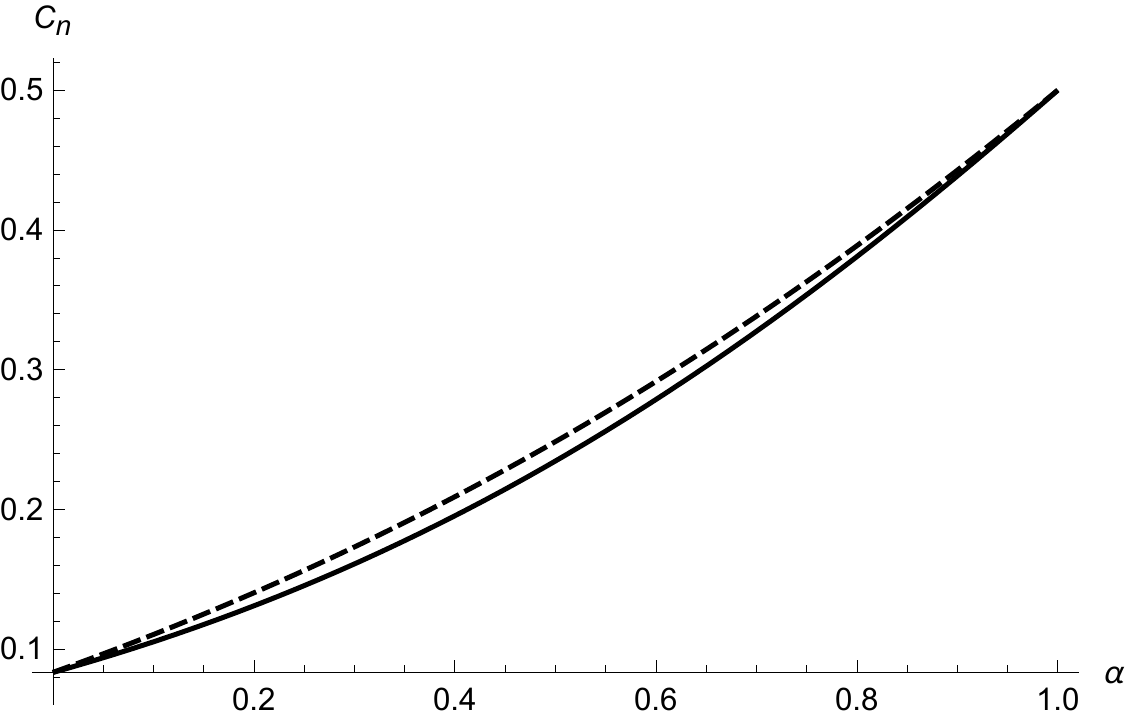}
\includegraphics[scale = 0.75]{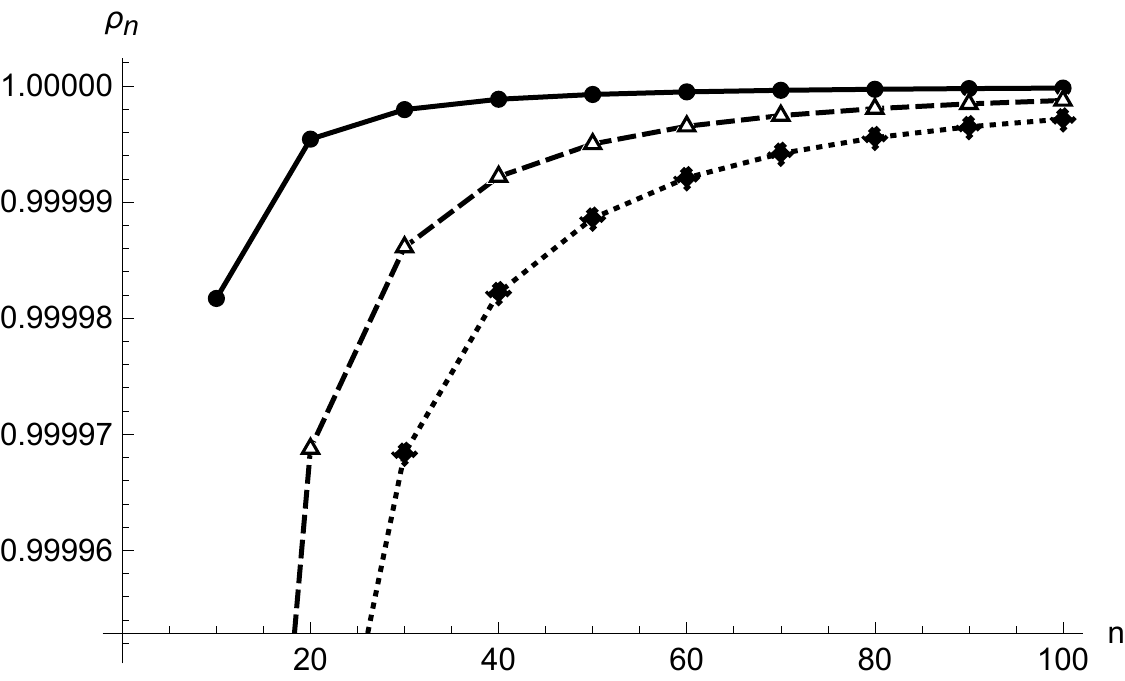}
\caption{On the left: comparison of our result (\ref{eqn:L1SchemeOptimalC}) (solid line) with the bound found in \cite{Li19a} (dashed line). On the right: the value of $\rho_n$ as in (\ref{eqn:L1SchemeRho}) for different $n$ and $\alpha = 0.1$ (solid line, circles), $\alpha = 0.5$ (dashed line, triangles), and $\alpha=0.9$ (dotted line, crosses). }
\label{fig:L1Scheme}
\end{figure}

\begin{rem}
	\label{rem:Truncation}
	The above theorem concerns only the regular and smooth case of $y$ being a $C^2[0,T]$ function and a uniform mesh. When lower regularity is assumed, only a slower order of convergence can be attained. For example, an important regularity assumption is the continuity with singular behavior of derivatives. That is, provided that $|y^{(m)}(t)|\leq C(1+t^{\alpha-m})$ for $m=0,1,2$, the following estimate of the truncation error was proved in \cite{stynes2017error} 
	\begin{equation}
		\epsilon_n(h) \leq C n^{-\min\{2-\alpha,r \alpha\}}, \quad t_n = T \left(\frac{n h}{T}\right)^r,
	\end{equation}
	Note that it gives a general case for a graded mesh. Our case, the uniform grading, corresponds to $r = 1$. Observe that by choosing an appropriate mesh, we can obtain high orders of convergence. Specifically, choosing $r = (2-\alpha)/\alpha$ yields an optimal grading in which the order is the same as in the smooth case, that is, $2-\alpha$. However, in that case, the mesh becomes clustered for small $\alpha$.
	
	Note also that this is independent of $h$. We will use this estimate below in the convergence proof. 
\end{rem}

In exactly the same spirit, we can conduct an analysis for the remainder of the discretization of the fractional integral (\ref{eqn:FracInt}). However, the result is of a different nature. 

\begin{prop}
Let $y\in C^1[0,T]$ and $0<\alpha<1$. Then for any $t_n\rightarrow t\in(0,T]$ as $n\rightarrow \infty$ we have
\begin{equation}
	I^\alpha y(t_n) = J^\alpha y(t_n) +  h \, D_n y'(\tau),
\end{equation}
where the discretized fractional integral is
\begin{equation}
	\label{eqn:FracIntDisc}
	J^\alpha y(t_n) := \frac{h^\alpha}{\Gamma(1+\alpha)} \sum_{i=1}^n b_{n-i}(\alpha) y(t_i)
\end{equation}
with some $\tau \in (0,T]$ and the weights $b_i$ defined in (\ref{eqn:Weights}). Moreover, we have the asymptotic expansion
\begin{equation}
	\label{eqn:FracIntErrorConstant}
	D_n \sim -\frac{t^\alpha}{\Gamma(1+\alpha)} \left(\frac{1}{2} +\frac{\zeta(-\alpha)}{n^\alpha} + \frac{\alpha}{12n}\right) \quad \text{as} \quad n \rightarrow\infty, \quad h \rightarrow 0.
\end{equation}
\end{prop}
\begin{proof}
By the same subdivision of the $(0,t_n)$ interval as in the above proof and by a simple Taylor expansion in each subinterval $(t_{i-1},t_i)$ 
\begin{equation}
	y(s) = y(t_i) + y'(\xi_i) (s-t_i),	\quad s,\xi_i \in  (t_{i-1},t_i),
\end{equation}
we can immediately write
\begin{equation}
	I^\alpha y(t_n) = \frac{h^\alpha}{\Gamma(\alpha)} \sum_{i=1}^n b_{n-i}(\alpha) y(t_i) + \frac{1}{\Gamma(\alpha)} \sum_{i=1}^n\int_{t_{i-1}}^{t_i} (t_n - s)^{\alpha-1} (s-t_i) y'(\xi_i) ds.
\end{equation}
Since the integrand is negative we can use the mean value theorem and take the derivative first outside the integral and then outside, which produces some $\tau \in (0,T]$. Further, by a simple calculation, we have
\begin{equation}
	\int_{t_{i-1}}^{t_i} (t_n - s)^{\alpha-1} (s-t_i) ds = -\frac{h^{1+\alpha}}{\alpha(1+\alpha)} \left((n-i)^{1+\alpha} - (n-i-\alpha)(n-i+1)^\alpha\right),
\end{equation}
and hence our error constant $D_n$ is
\begin{equation}
	D_n = \frac{h^\alpha}{\Gamma(2+\alpha)} \sum_{i=1}^n \left((n-i)^{1+\alpha} - (n-i-\alpha)(n-i+1)^\alpha\right).
\end{equation}
We can write $n - i - \alpha = (n-i + 1) - (1+\alpha)$ which yields
\begin{equation}
	D_n = \frac{h^\alpha}{\Gamma(\alpha)} \left[ \sum_{i=1}^n \left((n-i)^{1+\alpha} - (n-i+1)^{1+\alpha}\right) + (1+\alpha)  \sum_{i=1}^n (n-i+1)^\alpha \right]. 
\end{equation}
The first sum is telescoping, while in the second, we can introduce another summation variable $j = n-i+1$
\begin{equation}
	D_n = \frac{h^\alpha}{\Gamma(2+\alpha)} \left[ -n^{1+\alpha} + (1+\alpha)  \sum_{i=1}^n j^\alpha \right],
\end{equation}
which brings us to the harmonic number of order $-\alpha$. Using its asymptotic form as in (\ref{eqn:HarmonicNumberAsymptotics}) leads to the following
\begin{equation}
	D_n \sim \frac{h^\alpha}{\Gamma(2+\alpha)} \left[ -n^{1+\alpha} + n^{1+\alpha} + \frac{1+\alpha}{2} n^\alpha + \frac{\alpha(1+\alpha)}{12}\frac{1}{n^{1-\alpha}}  + (1+\alpha)\zeta(-\alpha)\right], 
\end{equation}
as $n\rightarrow \infty$. Since also $n h \rightarrow t$ the conclusion follows. 
\end{proof}
As we can see, the remainder of the simple discretization of the fractional integral is of integer order. For a numerical illustration of the above, we choose $t = T = 1$ and $y(t) = t$ for the first integral to be constant. In Fig. \ref{fig:FracInt} we plotted the value of
\begin{equation}
\label{eqn:rhonTilde}
\widetilde{\rho}_n := -\left|\frac{I^\alpha y(1) - J^\alpha y(t_n)}{\frac{1}{n}} + \frac{1}{2\Gamma(1+\alpha)}\right| \left(\frac{\zeta(-\alpha)}{n^\alpha} + \frac{\alpha}{12}\right)^{-1},
\end{equation}
for different values of $\alpha$. According to (\ref{eqn:FracIntErrorConstant}) the above expression should converge to $1$ with increasing $n$. As we can see, the convergence is very fast for the entire range of $\alpha \in (0,1)$. We can also observe that for $\alpha \rightarrow 1^-$ the two last terms in (\ref{eqn:FracIntErrorConstant}) cancel each other (since $\zeta(-1) = -1/12$), and the expansion becomes exact (not asymptotic). Furthermore, we can see that the order of discretization is independent of $\alpha$ and equals $1$ in contrast to $2-\alpha$ of the L1 scheme for the Caputo derivative. 

\begin{figure}
\centering
\includegraphics[scale = 1.2]{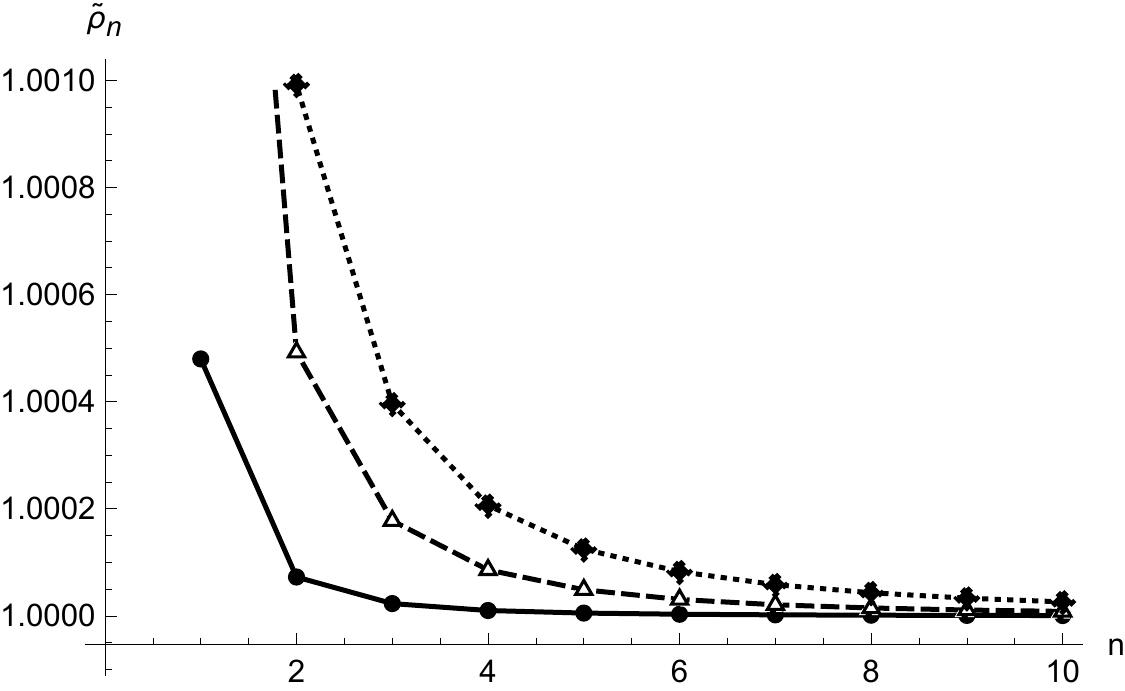}
\caption{The value of (\ref{eqn:rhonTilde}) as a function of $n$ for $\alpha = 0.1$ (solid line, circles), $\alpha = 0.5$ (dashed line, triangles), and $\alpha = 0.9$ (dotted line, crosses)}
\label{fig:FracInt}
\end{figure}

\section{Galerkin method}
Having prepared the discretization of the Caputo derivative, we proceed to the main theme of this paper. First, we set the framework. Throughout the following, a generic constant that can depend on $u$ and all of its derivatives will be denoted by $C$. We allow it to change its value even in the same chain of inequalities. We will work in the $L^2(0,1)$ Hilbert space endowed with the standard norm $\left\|\cdot\right\|$. The general Sobolev space of $m$-th differentiable will be denoted by $H^m(0,1)$ and we equip it with the usual norm
\begin{equation}
\left\|f\right\|_m = \left(\sum_{i=0}^m \left\|\frac{\partial^i f }{\partial x^i}\right\|^2\right)^\frac{1}{2}
\end{equation}
Due to the Poincar\'e-Friedrichs inequality, the space $H_0^1(0,1)\subseteq H^1(0,1)$ of functions with vanishing trace is also a Sobolev space with the norm $\left\|\cdot\right\|_1$. In what follows, we usually will omit denoting the interval $(0,1)$ when addressing a particular space, i.e. we will write $H^m$ instead of $H^m(0,1)$. 

In order to propose a Galerkin method, we recast the main equation (\ref{eqn:MainPDE}) into its weak form. By multiplying it by the test function $v\in H_0^1$ and integrating by parts, we obtain
\begin{equation}
\label{eqn:MainPDEWeak}
\left(\partial^\alpha_t u, v\right) + a(D(u); u, v) = (f(t,u), v), \quad v \in H_0^1,
\end{equation}
with the initial condition
\begin{equation}
u(0) = \varphi. 
\end{equation}
Here, the functional $a$ is defined as
\begin{equation}
\label{eqn:FunctionalA}
a(D(w);u,v) = \int_0^1 D(w(x)) u_x(x) v_x(x) dx. 
\end{equation}
For well-posedness, we assume boundedness and smoothness of $D$ and $f$
\begin{equation}
\label{eqn:Assumptions}
0<D_- \leq D(u) \leq D_+, \quad |f(x, t, u)| \leq F, \quad |D'(u)| + |f_u(x,t,u)| \leq L, 
\end{equation}
where $D_\pm$, $F$, and $L$ are positive constants. We will seek an approximate solution of (\ref{eqn:MainPDEWeak}) in a suitable $N$-dimensional space $V_N\subseteq H_0^1$. Furthermore, we assume that $V_N$ is either a trigonometric or algebraic polynomial space in order to utilize spectral accuracy. More specifically, if $P_N$ is the $L^2$ projection onto $V_N$ we have \cite{Can07}
\begin{equation}
\left\|u - P_n u\right\| \leq C N^{-m} \left\|u\right\|_m, \quad \left\|u - P_n u\right\|_l \leq C N^{2l-\frac{1}{2}-m}  \left\|u\right\|_m, \quad u \in H^m, \quad m\geq 1, \quad 1\leq l\leq m.
\end{equation}
Therefore, for an infinitely smooth function, the spectral approximation converges faster than any power of $N$. We also assume that our finite-dimensional space enjoys a form of inverse inequality for norms and for derivatives
\begin{equation}
\label{eqn:InverseInequality}
\left\|v_x\right\|_\infty \leq C N^\mu \left\|v_x\right\| \leq C N^{\mu+\nu} \left\|v\right\|, \quad \nu,\mu > 0.	
\end{equation} 
For example, for the Legendre polynomial space, we have $\nu = 1$ and $\mu = 2$ (for a comprehensive discussion, see \cite{Can07}). Further, instead of the above orthogonal projection, we will mainly use the more useful Ritz projection $R_N u$ defined as a unique solution of the following linear elliptic problem (see \cite{Tho07})
\begin{equation}
\label{eqn:RitzProjection}
a(D(u); R_N u  - u, v) = 0, \quad v \in V_N,
\end{equation} 
where $u$ is given. It is also a well-known fact that $R_N$ enjoys the same spectral accuracy as $P_N$
\begin{equation}
\label{eqn:RitzProjectionAccuracy}
\left\|u - R_N u\right\| \leq C N^{-m} \left\|u\right\|_m, \quad \left\|u - R_N u\right\|_l \leq C N^{2l-\frac{1}{2}-m}  \left\|u\right\|_m, \quad u \in H^m.
\end{equation}
This fact can be seen by adapting the classical reasoning for the finite element method to the spectral setting (see for ex. \cite{Tho07}). Moreover, using the inverse inequalities (\ref{eqn:InverseInequality}) and the above accuracy estimates we can show that for $u\in H^m$ with $m$ sufficiently large, we have \cite{Tho07,Plo21}
\begin{equation}
\label{eqn:RitzInftyEstimate}
\begin{split}
	\left\|(R_N u)_x\right\|_\infty 
	&\leq \left\|(R_N u -u)_x\right\|_\infty + \left\|u_x\right\|_\infty \leq C N^{\mu+\nu} \left\|R_N u -u\right\| + \left\|u_x\right\|_\infty \\
	&\leq C N^{\mu+\nu-m} \left\|u\right\|_m + \left\|u_x\right\|_\infty \leq C \left\|u_x\right\|_\infty, 
\end{split}
\end{equation}
which will be useful in the sequel. Furthermore, concerning the interaction of the Caputo derivative and the $L^2$ scalar product, we have the following result. It was previously proved, for example, in \cite{Li18}, but here we present a different proof for the continuous case. 
\begin{prop}\label{prop:CaputoScalar}
If $y^n \in L^2$ for any $n\in\mathbb{N}$, then
\begin{equation}
	\label{eqn:CaputoScalar}
	\left(\delta^\alpha y^n, y^n\right) \geq \frac{1}{2} \delta^\alpha \left\|y^n\right\|^2.
\end{equation}
Similarly, for $y \in C([0,T]; L^2)$ we have
\begin{equation}
	\left(\partial_t^\alpha y(t), y(t)\right) \geq \frac{1}{2} \partial^\alpha \left\|y(t)\right\|^2.
\end{equation}
\end{prop}
\begin{proof}
Let $y^n \in L^2$. Then, by expanding the L1 scheme (\ref{eqn:L1Scheme2}) can write
\begin{equation}
	\begin{split}
		\left(\delta^\alpha y^n, y^n\right) &= (b_0 y^n - b_{n-1} y^0 - \sum_{i=1}^{n-1} (b_{n-i-1}-b_{n-i})y^i, y^n) \\
		&= b_0 \left\|y^n\right\|^2 - b_{n-1} (y^0,y^n) - \sum_{i=1}^{n-1} (b_{n-i-1}-b_{n-i}) (y^i, y^n).
	\end{split}
\end{equation} 
Now, by the Cauchy inequality $a b \leq (a^2+b^2)/2$ we have the following
\begin{equation}
	\left(\delta^\alpha y^n, y^n\right) \geq \left(b_0 - \frac{1}{2}b_{n-1}\right) \|y^n\|^2 - \frac{1}{2}b_{n-1}\|y^0\|^2 - \frac{1}{2} \sum_{i=1}^{n-1} (b_{n-i-1}-b_{n-i}) \|y^i\|^2 - \frac{1}{2}\|y^n\|^2 \sum_{i=1}^{n-1} (b_{n-i-1}-b_{n-i}).
\end{equation}
But, by the definition of weights (\ref{eqn:Weights}) we have $\sum_{i=1}^{n-1} (b_{n-i-1}-b_{n-i}) = b_0-b_{n-1}$, and hence
\begin{equation}
	\left(\delta^\alpha y^n, y^n\right) \geq \left(b_0 - \frac{1}{2}b_{n-1} - \frac{1}{2}b_0 - \frac{1}{2} b_{n-1}\right) \|y^n\|^2 -\frac{1}{2}b_{n-1}\|y^0\|^2 - \frac{1}{2} \sum_{i=1}^{n-1} (b_{n-i-1}-b_{n-i}) \|y^i\|^2.
\end{equation}
By canceling the terms, we arrive at (\ref{eqn:CaputoScalar}). 

For the continuous case, we fix $t\in(0,T)$ and introduce the grid $t_i = i h$ with $h = t/n$. By writing $y^i = y(t_i)$ and using (\ref{eqn:L1Error}) we have the following
\begin{equation}
	\partial^\alpha y(t) = \delta^\alpha y^n - \epsilon_n(h),
\end{equation}
where $R_n(h) \rightarrow 0^+$ as $h\rightarrow 0^+$. Now, by the previous calculation concerning the discrete case
\begin{equation}
	\left(\partial_t^\alpha y(t), y(t)\right) = \left(\delta_t^\alpha y(t), y(t)\right) - \left(\epsilon_n(h), y(t)\right) \geq \frac{1}{2} \delta^\alpha \|y^n\|^2 - \left(\epsilon_n(h), y(t)\right),
\end{equation}
and using (\ref{eqn:L1Error}) with $\|y\|^2$ instead of $y$ with possibly different error $\widetilde{\epsilon}_n(h)$ we have
\begin{equation}
	\left(\partial_t^\alpha y(t), y(t)\right) \geq \frac{1}{2} \partial^\alpha \|y(t)\|^2 + \frac{1}{2} \|\widetilde{\epsilon}_n(h)\|^2 - \left(\epsilon_n(h), y(t)\right).
\end{equation}
The proof is concluded with the refinement of the grid $h\rightarrow 0^+$, that is, $n\rightarrow\infty$. 
\end{proof}

The Galerkin spectral method can now be formulated by averaging (\ref{eqn:MainPDEWeak}) over the finite-dimensional space $V_N$ along with using the L1 scheme (\ref{eqn:L1Scheme}) for the Caputo derivative. However, to reduce computational complexity, we linearize the method using extrapolation \cite{Tho07}. That is, for any regular function of time $y=y(t)$, we define
\begin{equation}
\label{eqn:Extrapolation}
\widehat{y}(t_{n}) := 2 y(t_{n-1}) - y(t_{n-2}),
\end{equation}
which is of second-order accuracy, that is $|y(t_n) - \widehat{y}(t_n)| \leq C h^2$. Finally, we propose the following fully discrete method to calculate $U^n\in V_N$ which is a numerical approximation of $u(t_n)$ with $t_n$ defined in (\ref{eqn:TimeMesh})
\begin{equation}
\label{eqn:Galerkin}
(\delta^\alpha U^n, v) + a(D(\widehat{U}^n); U^n, v) = (f(t_n, \widehat{U}^n), v), \quad v\in V_N, \quad n \geq 2,
\end{equation}
with the initial condition
\begin{equation}
U^0 = P_N \varphi. 
\end{equation}
As we can see, the arguments of diffusivity $D$ and the source $f$ have been extrapolated, rendering the scheme linear. Since the extrapolation (\ref{eqn:Extrapolation}) requires two time steps back, we initialize the scheme with a \emph{single} solution of the nonlinear problem
\begin{equation}
\label{eqn:GalerkinInitialization}
(\delta^\alpha U^1, v) + a(D(U^1); U^1, v) = (f(t_1, \widehat{U}^1), v), \quad v\in V_N,
\end{equation} 
which can, for example, be solved by a multiple predictor-corrector iteration or Newton's method (see \cite{Li18}). Before we proceed to the main results of this paper, we recall an appropriate version of the discrete fractional Gr\"onwall's lemma. In \cite{Lia18} authors proved its very general version. For clarity we state here only its special version, which will be needed in the following.
\begin{lem}[Discrete fractional Gr\"onwall inequality \cite{Lia18}\}]\label{lem:Gronwall}
	Let $y \in C[0,T]$ and $\left\{t_n\right\}$ be the mesh (\ref{eqn:TimeMesh}) with $t_M = T$ and a sufficiently small step $h>0$. If for some constants $\lambda_{0,1,2}$, $A$, and $F$ we have
	\begin{equation}
		\delta^\alpha y^n \leq \lambda_0 y^n + \lambda_1 y^{n-1} + \lambda_2 y^{n-2} + \frac{A}{n^\alpha} + F, \quad 0<\alpha<1, \quad n = 1,2,...
	\end{equation}
	then
	\begin{equation}
		y_n \leq 2 E_\alpha(2 \lambda T^\alpha) \left(v^0 + A\tau^\alpha + \frac{T^\alpha}{\Gamma(1+\alpha)}F\right), \quad n = 1,2,...,M,
	\end{equation}
	where $\lambda = \lambda_0 + \lambda_1+\lambda_2$, $E_\alpha(t) = \sum_{i=0}^\infty z^k/\Gamma(1+k\alpha)$ is the Mittag-Leffler function \cite{Li19}. 
\end{lem}

The first result we prove is the unconditional stability of the proposed scheme.

\begin{thm}[Stability]
Let $U^n$ be the solution of (\ref{eqn:Galerkin})-(\ref{eqn:GalerkinInitialization}). Then
\begin{equation}
	\|U^n\|\leq C\left(\|U^0\|+F\right),
\end{equation}
where $F$ is the bound for $f$ as in (\ref{eqn:Assumptions}). 
\end{thm}
\begin{proof}
We put $v = U^n$ in (\ref{eqn:Galerkin}) to obtain
\begin{eqnarray}
	(\delta^\alpha U^n, U^n) + a(D(\widehat{U}^n); U^n, U^n) = (f(t_n, \widehat{U}^n), U^n), \quad n\geq 2.
\end{eqnarray}
By our assumption of diffusivity (\ref{eqn:Assumptions}) the form $a$ is positive-definite.
\begin{equation}
	a(D(\widehat{U}^n); U^n, U^n) \geq D_- \|U^n_x\| > 0.
\end{equation}
Furthermore, by Proposition \ref{prop:CaputoScalar}, we can write
\begin{equation}
	\frac{1}{2} \delta^\alpha \|U^n\|^2 \leq F \|U^n\| \leq \frac{1}{2} \left(F^2 + \|U^n\|^2\right).
\end{equation}
By the same reasoning, this estimate can be obtained exactly for $n=1$. Finally, applying the fractional Gr\"onwall's lemma (Lemma \ref{lem:Gronwall}) to $y^n = \|U^n\|^2$ we conclude that
\begin{equation}
	\|U^n\|^2 \leq C \left(\|U^0\|^2 + \frac{T^\alpha}{\Gamma(1+\alpha)}F^2\right).
\end{equation}
Using the inequality $a^2 + b^2 \leq (a+b)^2$, we arrive at the conclusion. 
\end{proof}

We are now ready to proceed with the convergence result. Note the time regularity of the solution. It is widely known that the solution of the linear constant or the space-dependent coefficient subdiffusion behaves as $t^\alpha$ for small times \cite{Sak11,McL10} (that is, its derivative blows up at the origin). The time-dependent coefficient case is still being investigated \cite{Jin19a,Mus18}, let alone the semilinear version \cite{al2019numerical}. In the following theorem, we consider two regularity cases: $C^2$ smooth and the typical power-type behavior near the origin. 

\begin{thm}[Convergence]\label{thm:Convergence}
	Let $u \in C([0,T]; H^m)$ be a solution of (\ref{eqn:MainPDEWeak}) and $U^n$ be a solution of (\ref{eqn:Galerkin}). Assume that $u_x$ is bounded. Then, for sufficiently large $m$ and small $h>0$ we have
	\begin{equation}
		\|u(t_n) - U^n \| \leq C \left(N^{-m} + E(h)\right),
	\end{equation}	
	where the constant $C$ depends on $\alpha$ and derivatives of $u$, while $E(h)$ is the temporal error that depends on the regularity of the solution
	\begin{equation}
		\label{eqn:TimeError}
		E(h) = \begin{cases}
			h^\alpha, & \|\partial^{(m)}_t u(t)\|\leq C(1+t^{\alpha-m}), \\
			h^{2-\alpha}, & \|\partial^{(m)}_t u(t)\|\leq C, 
		\end{cases}
		\quad m = 0, 1, 2.
	\end{equation}
\end{thm}

\begin{proof}
We start by decomposing the error into two parts
\begin{equation}
	u(t_n) - U^n = u(t_n) - R_N u(t_n) + R_N u(t_n) - U^n =: r^n + e^n.
\end{equation}
Since the estimate on $r^n$ is provided in (\ref{eqn:RitzProjectionAccuracy}) we can focus only on $e^n$ due to the fact that
\begin{equation}
	\|u(t_n) - U^n\| \leq \|r^n\| + \|e^n\| \leq C N^{-m} + \|e^n\|.
\end{equation} 
The error equation can be devised as follows. First, for any $v\in V_N$, we write
\begin{equation}
	\begin{split}
		\left(\delta^\alpha e^n, v\right) + a(D(\widehat{U}^n); e^n, v) 
		&= \left(\delta^\alpha R_N u(t_n), v\right) - \left(\delta^\alpha U^n, v\right) + a(D(\widehat{U}^n); R_N u(t_n), v) - a(D(\widehat{U}^n); U^n, v) \\
		&= \left(\delta^\alpha R_N u(t_n), v\right) + a(D(\widehat{U}^n); R_N u(t_n), v) - (f(t_n,\widehat{U}^n), v),
	\end{split}
\end{equation}
where we have used the fact that $U^n$ is a solution of (\ref{eqn:Galerkin}). Next, we can add and subtract terms in the form $a$ and utilize the weak form of the solution (\ref{eqn:MainPDEWeak})
\begin{equation}
	\begin{split}
		\left(\delta^\alpha e^n, v\right) &+ a(D(\widehat{U}^n); e^n, v) 
		= \left(\delta^\alpha R_N u(t_n), v\right) + a(D(\widehat{U}^n)-D(u(t_n)); R_N u(t_n), v) \\
		&+ a(D(u(t_n)); R_N u(t_n), v) - (f(t_n,\widehat{U}^n), v) \\
		&= \left(\delta^\alpha R_N u(t_n) - \partial^\alpha_t u(t_n), v\right) + a(D(\widehat{U}^n)-D(u(t_n)); R_N u(t_n), v) \\
		&- (f(t_n,\widehat{U}^n)-f(t_n, u(t_n)), v).
	\end{split}
\end{equation} 
When we put $v = e^n$ and use our assumptions (\ref{eqn:Assumptions}) along with Poincar\'e-Friedrichs inequality and Proposition \ref{prop:CaputoScalar} we obtain the following
\begin{equation}
	\frac{1}{2}\delta^\alpha \|e^n\|^2 + D_0 \|e^n\|_1 \leq \rho_1 + \rho_2 + \rho_3, 
\end{equation}
where the definitions of the remainders $\rho_i$ are understood and $D_0$ is some constant. Now, we proceed with estimates of these. By Theorem \ref{thm:L1Scheme} and (\ref{eqn:RitzProjectionAccuracy}) we have
\begin{equation}
	\begin{split}
		\rho_1 &\leq \|\delta^\alpha R_N u(t_n) - \partial^\alpha_t u(t_n) \| \|e^n\| \\
		&\leq  \left(\|\delta^\alpha r^n \| + \| \delta^\alpha u(t_n) - \partial^\alpha_t u(t_n)\| \right)\|e^n\| \leq C (N^{-m} + \epsilon_n(h)) \|e^n\|_1,
	\end{split}
\end{equation}
where the discretization error of the Caputo derivative $\epsilon_n(h)$ defined in (\ref{eqn:L1Error}) behaves as $h^{2-\alpha}$ for a $C^2$ smooth function (Theorem \ref{thm:L1Scheme}), while as $n^{-\alpha}$ for a power-law type behavior (Remark \ref{rem:Truncation}). Next, the second remainder can be estimated with the use of (\ref{eqn:RitzInftyEstimate}) and the regularity of $D$ as assumed in (\ref{eqn:Assumptions})
\begin{equation}
	\rho_2 \leq C \|(R_N u(t_n))_x\|^2_\infty \|\widehat{U}^n - u(t_n) \| \|e^n\|_1 \leq C\left(\|\widehat{U}^n - \widehat{u}(t_n)\| + \|\widehat{u}(t_n) - u(t_n)\|\right) \|e^n\|_1,
\end{equation}
Now, by the definition of the extrapolation (\ref{eqn:Extrapolation}) we have
\begin{equation}
	\rho_2 \leq C \left(\|r^{n-1}\| + \|r^{n-2}\| + \|e^{n-1}\| + \|e^{n-2}\| + h^2\right) \|e^n\|_1 \leq C \left(N^{-m} +\|e^{n-1}\| + \|e^{n-2}\| + h^2 \right).
\end{equation}
And the last remainder, due to the regularity assumptions (\ref{eqn:Assumptions}), satisfies
\begin{equation}
	\rho_3 \leq L \|\widehat{U}^n - u(t_n) \| \|e^n\|_1 \leq C \left(N^{-m} +\|e^{n-1}\| + \|e^{n-2}\| + h^2 \right) \|e^n\|_1.
\end{equation}
Putting all the above estimates back to the error estimate and using the $\epsilon$-Cauchy inequality, i.e. $a b \leq (\epsilon a^2 + \epsilon^{-1}b^2)/2$, to extract the $H^1$ norm we obtain the following
\begin{equation}
	\frac{1}{2}\delta^\alpha \|e^n\|^2 + D_0 \|e^n\|_1^2 \leq C\left(\|e^{n-1}\|^2 + \|e^{n-2}\|^2 + \left(N^{-m} + \epsilon_n(h)\right)^2\right) + D_0 \|e^n\|_1^2.
\end{equation}
Therefore,
\begin{equation}
	\delta^\alpha \|e^n\|^2 \leq C\left(\|e^{n-1}\|^2 + \|e^{n-2}\|^2 + \left(N^{-m} + \epsilon_n(h)\right)^2\right),
\end{equation}
and the discrete fractional Gr\"onwall lemma (Lemma \ref{lem:Gronwall}) yields
	\begin{equation}
		\|e^n\|^2 \leq C \left(\|e^0\|^2 + \left(N^{-m} + E(h)\right)^2\right),
	\end{equation}
	with $E(h)$ defined in (\ref{eqn:TimeError}). Note that when applying Lemma \ref{lem:Gronwall} we have used the case $A=0$ for $C^2$ smooth solution while for the typical power-law behavior we have taken $F=0$ and used the result cited from Remark \ref{rem:Truncation}. Finally, observe that $\|e^0 \|= \|R_N u(0) - U^0 \|= \|R_N \phi - P_N \phi\|\ \leq C N^{-m}$ what finishes the proof.
\end{proof}

\begin{rem}\label{rem:Error}
	For the solution satisfying a typical power-type behavior at the origin, we have proved that the scheme error for our quasilinear equation is $O(h^\alpha)$ for $h\rightarrow 0^+$. However, we might presume that this estimate is not optimal. Known results for linear and semilinear subdiffusion equations \cite{Kop19,al2019numerical} indicate that 
	\begin{equation}
		E(h) = E_n(h) = C t_n^{\alpha-1} h,
	\end{equation}
	That is, the error depends on the time instant. Near the origin, the order is $\alpha$, while at fixed $t>0$ it is $1$. More precisely, when $h\rightarrow 0^+$ we have $\max_{t\in (0,T]} E(h) = O(h^\alpha)$ (put $t=t_1=h$). On the other hand, when we compute the error at a fixed time $t_n \rightarrow t > 0$ we obtain $R_n(h) = O(h)$. We verify this behavior numerically for our equation (\ref{eqn:MainPDE}) in the next section. Nevertheless, we would like to indicate that proving an optimal error bound for the solution of the quasilinear subdiffusion equation with the L1 discretization of the Caputo derivative is still an open problem (both for smooth and nonsmooth initial data). Some initial steps have been taken in \cite{plociniczak2022error}. 
	
	Note also that all the steps in the proof of Theorem \ref{thm:Convergence} could have been done for a graded mesh $t_n = T (n h /T)^r$ as it would only affect the error term $\epsilon_n(h)$.  
\end{rem}

\begin{rem}\label{rem:Boundedness}
By a standard argument \cite{Tho07} the global boundedness of both the source $f$ and the diffusivity $D$ as in (\ref{eqn:Assumptions}) can be relaxed to hold only locally. This follows from the fact that $U$ is generally close to $u$ and therefore is affected only by local values of $f$ and $D$. 
\end{rem}

\section{Implementation and numerical illustration}
In this section, we discuss the practical implementation of the Galerkin scheme (\ref{eqn:Galerkin})-(\ref{eqn:GalerkinInitialization}). Let $\left\{\Phi_k\right\}_{k=1}^N$ be the basis of $V_N$. To satisfy the boundary conditions, we can choose the modal Legendre basis \cite{Can07}
\begin{equation}
\Phi_k(x) = L_k(2x-1) - L_{k+2}(2x-1),
\end{equation}
with Legendre polynomials $L_k$. A similar construction can also be done with Chebyshev polynomials or, more generally, Jacobi orthogonal polynomials (for a thorough discussion, see \cite{She11}). We expand the solution $U^n$ in this basis
\begin{equation}
U^n = \sum_{k=1}^N y_k^n \Phi_k.
\end{equation}
Taking $v = \Phi_j$, denoting $\textbf{y}^n = (y_1^n,...,y_N^n)$, and plugging the above into the Galerkin scheme (\ref{eqn:Galerkin}) we obtain
\begin{equation}
M \delta^\alpha \textbf{y}^n + A(\widehat{\textbf{y}}^n) \textbf{y}^n = \textbf{f}^n(\widehat{\textbf{y}}^n),
\end{equation}
where the mass matrix $M = \left\{M_{ij}\right\}_{i,j=1}^N$, the stiffness matrix $A = \left\{A_{ij}\right\}_{i,j=1}^N$, and the load vector $\textbf{f}^n=\left\{f_i^n\right\}_{i=1}^N$ are defined by
\begin{equation}
M_{ij} = (\Phi_i, \Phi_j), \quad A_{ij}(\textbf{y}^n) = a\left(D\left(\sum_{k=1}^N \widehat{y}_k^n \Phi_k\right); \Phi_i,\Phi_j\right), \quad f_i^n = \left(f\left(t_n, \sum_{k=1}^N \widehat{y}_k^n \Phi_k\right), \Phi_i\right).
\end{equation}
Note that in the case of an orthogonal basis, the mass matrix is diagonal. Finally, we can expand the Caputo derivative according to the L1 scheme (\ref{eqn:L1Scheme}) to arrive at
\begin{equation}
\label{eqn:GalerkinImplementation}
\begin{split}
	&\left(I + h^\alpha \Gamma(2-\alpha) M^{-1} A(\widehat{\textbf{y}}^n)\right) \textbf{y}^n \\
	&= b_{n-1} \textbf{y}^0 + \sum_{i=1}^{n-1} \left(b_{n-i-1}(1-\alpha) - b_{n-i}(1-\alpha)\right) \textbf{y}^i + h^\alpha \Gamma(2-\alpha) M^{-1} \textbf{f}^n(\widehat{\textbf{y}}^n),
\end{split}
\end{equation}
which clearly possesses a time-stepping form: the right-hand side depends on $\textbf{y}^i$ for $i=0,1,...,n-1$. The first step, that is, (\ref{eqn:GalerkinInitialization}), is implemented in a similar way with the distinction of a nonlinear algebraic equation. The nonlocality of the method is manifestly evident in the sum above. This is the reason for the larger computational and memory cost of all schemes for fractional equations. One way of dealing with this problem is to use a spectral method that requires only several degrees of freedom for spatial discretization in order to achieve good accuracy. Additionally, to maximize efficiency, we implement several mechanisms that reduce the calculation time. All programming has been done in Julia language, which is known to provide fast algorithm execution.
\begin{itemize}
\item All integrals are calculated with Gaussian quadrature with pregenerated nodes and weights. 
\item The mass matrix for our basis is sparse due to the compact combination of Legendre polynomials. 
\item The entries of the matrix $A$ and the vector $\textbf{f}^n$ can be calculated \emph{simultaneously}. We implement this by multi-threaded computations on several CPU cores. 
\item We check for linearity: if either $A$ or $\textbf{f}^n$ does not depend on $u$, it can be precalculated before the time-stepping. 
\end{itemize}

We choose two examples to test our method. The one is an "engineered" solution obtained by choosing an appropriate source function to obtain the desired solution. This approach has the advantage of straightforward error estimation. The second example does not have an exact analytic form solution and in order to estimate the error we use the extrapolation. In each of the following examples, we choose
\begin{equation}
D(u) = 1+u,
\end{equation}
which can be thought as a first-order Taylor expansion of a general smooth diffusivity. The source function in example 1 is chosen appropriately in order to produce the desired solution. The second example is a subdiffusive version of the Fisher-Kolmogorov equation with nonlinear diffusivity. For any fixed initial condition $u(x,0) = \varphi$, the examples are the following.
\begin{enumerate}
\item Exact solution with parameter $\mu > 0$
\begin{equation}
	\label{eqn:Exact}
	u(x,t) = (1+t^\mu) \varphi(x), \quad f(x,t,u) = \text{ chosen appropriately below}.
\end{equation}
\item Subdiffusive Fisher-Kolmogorov equation
\begin{equation}
	\label{eqn:FK}
	f(x,t,u) = u(1-u).
\end{equation}
\end{enumerate}
Note that the example 1 models the usual behavior of the solution of the subdiffusion equation \cite{Sak11}. When $0<\mu<1$ the derivative blows at $t=0$. The typical example is $\mu = \alpha$ (see \cite{Sak11}). In the below calculations we will consider the non-smooth case $\mu = \alpha$, smooth case $\mu = 2$, and the second example. In all examples, we choose the terminal time $T=1$. The initial condition $\varphi$ will be chosen differently for estimating the spatial and temporal order. This choice is due to the fact that we would like to decouple the errors from each other. 

Note that both the diffusivity $D$ and the source $f$ are unbounded in our examples. However, as noted in Remark \ref{rem:Boundedness} our results remain valid. 

We start by presenting estimates on the spatial order of convergence. The smooth initial condition and the source are the following
\begin{equation}
\label{eqn:XIC}
\varphi(x) = \sin(\pi x), \quad f(x,t,u) = \pi^2 \left(u(1-2u) - (1+t^\mu)^2\right) + \frac{\Gamma(1+\mu)}{\Gamma(1-\alpha+\mu)}t^{\mu-\alpha} \varphi(x).
\end{equation}
Note that it cannot be expressed as a finite combination of basis functions. In order to make our calculations independent of the temporal discretization for all examples, we compute the reference solution with step $h = 10^{-2}$ and number of degrees of freedom $N=50$. The error is calculated by comparing the solutions obtained with fixed $h$ and $3\leq N \leq 25$ with the reference one. In all the examples we have $\alpha = 0.5$. However, our calculations show that the overall shape of the error does not change for other values of $\alpha$. The semi-log plot of the error for our three examples is presented in Fig. \ref{fig:xOrder}. As we can see, for all of them the error decreases approximately linearly, indicating an exponential (spectral) accuracy of the method. For $N$ changing by only a few, the error falls by orders of magnitude. 

\begin{figure}
\centering
\includegraphics[scale = 0.7]{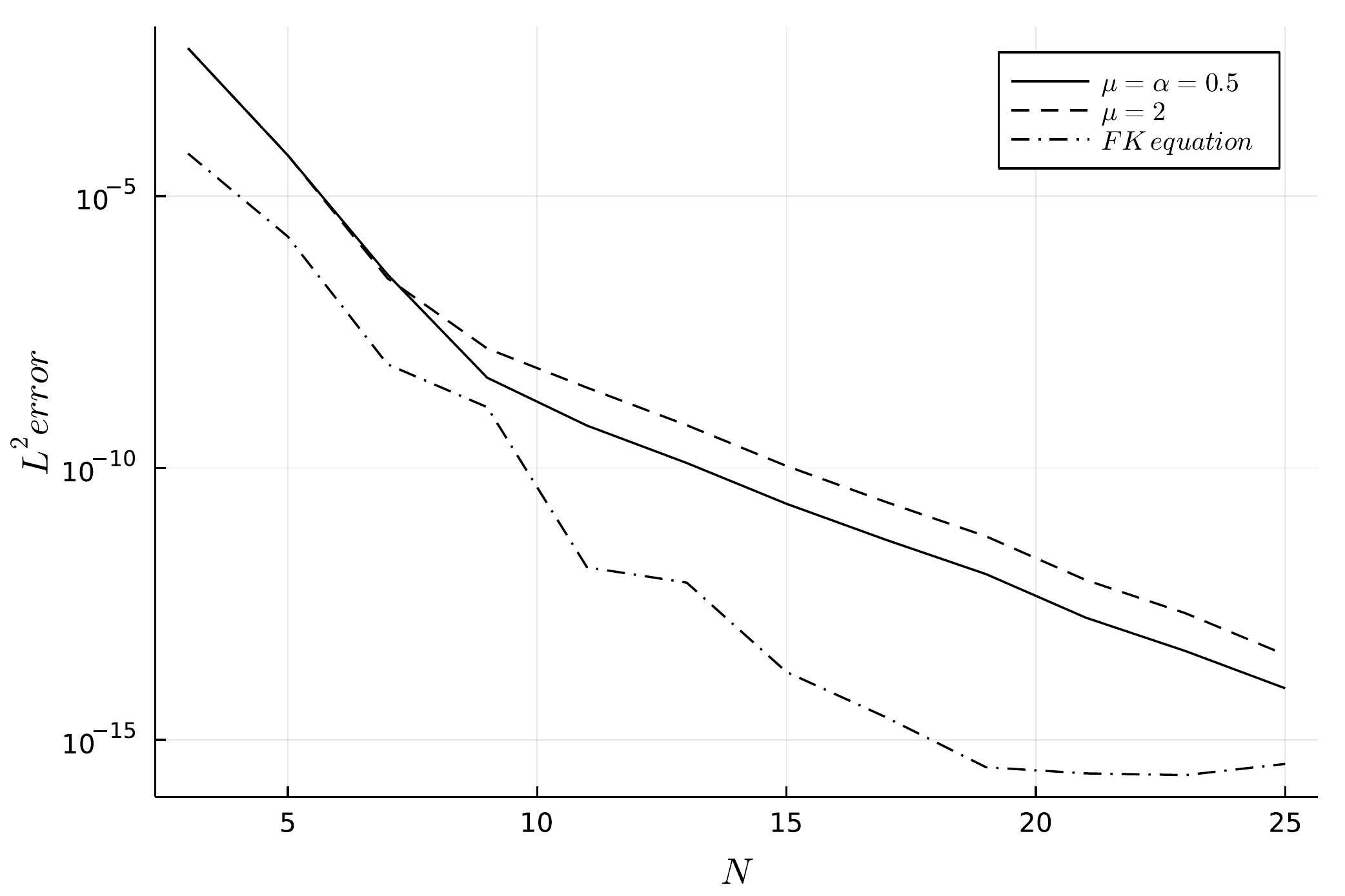}
\caption{A semi-log plot of the $L^2$ error for three examples: (\ref{eqn:Exact}) with $\mu = 2$ and $\mu = \alpha = 0.5$, and (\ref{eqn:FK}) as a function of $N$ with fixed $h = 10^{-2}$. The initial condition is taken to be (\ref{eqn:XIC}). }
\label{fig:xOrder}
\end{figure}

In the next simulation, we estimate the temporal convergence order of the method. The initial condition and the source are chosen to be
\begin{equation}
\label{eqn:TIC}
\varphi(x) = \frac{1}{6}\Phi_1(x) = x(1-x), \quad f(x,t,u) = (1+t^2)(1-t^2+6u) + \frac{\Gamma(1+\mu)}{\Gamma(1-\alpha+\mu)}t^{\mu-\alpha} \varphi(x)
\end{equation}
that is, the first basis function. Thanks to that choice, we calculate exactly in space and focus only on the temporal error. We set the number of degrees of freedom at $N=5$ and vary the time step $h$. When the exact solution is available, we calculate the error accordingly. Taking into account Remark \ref{rem:Error} we compute the order in two ways by evaluating either the maximal error in time (denoted by subscript $\infty$) or at a fixed point. In the Fisher-Kolmogorov equation, we estimate the convergence error with Aitken's formula based on extrapolation (see \cite{Plo17b})
\begin{equation}
	\label{eqn:Aitken}
	\begin{split}
		&\text{order} \approx \log_2 \frac{\|u_{h/2}(t)-u_h(t)\|}{\|u_{h/4}(t)-u_{h/2}(t)\|}, \quad t = 1 \text{ fixed}, \\
		&\text{order}_\infty \approx \log_2 \frac{\max_{t\in (0,1]}\|u_{h/2}(t)-u_h(t)\|}{\max_{t\in (0,1]}\|u_{h/4}(t)-u_{h/2}(t)\|}. 
	\end{split}
\end{equation}
where $u_h(t)$ is the solution calculated at $t \in (0,1]$ with step $h$. The results for (\ref{eqn:Exact}) are presented on the log-log plot shown in Figs. \ref{fig:tOrder}-\ref{fig:tOrderMax}.All the graphs become approximately parallel to the respective reference lines for the errors calculated at a fixed time. That is, the scheme for the smooth case attains an order $2-\alpha$, while for typical power-type behavior, we obtain the first order of convergence. However, for small $\alpha$, we observed somewhat lower accuracy, which would require one to use even finer grids to fully resolve. As for the maximal error over the time interval, we observe somewhat different behavior. That is, numerical calculations indicate that the order is indeed equal to $\alpha$ as Theorem \ref{thm:Convergence} claims. This might suggest that the actual temporal error of the method is of the order $t^{\alpha-1}_n h$, as speculated in Remark \ref{rem:Error}. Of course, this behavior naturally coincides with the linear or semilinear cases investigated in the literature.  

\begin{figure}
\centering
\includegraphics[scale = 0.7]{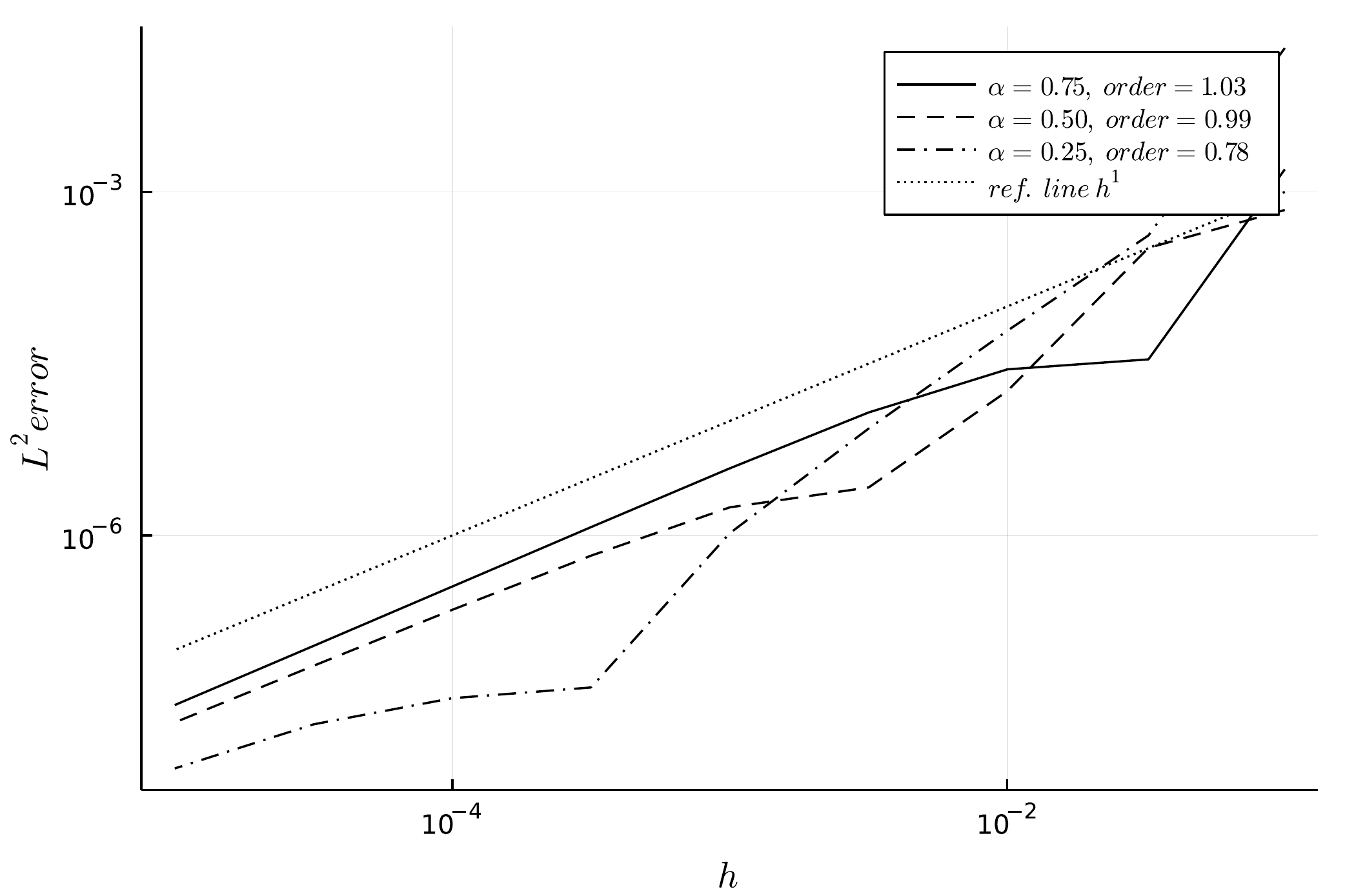}
\includegraphics[scale = 0.7]{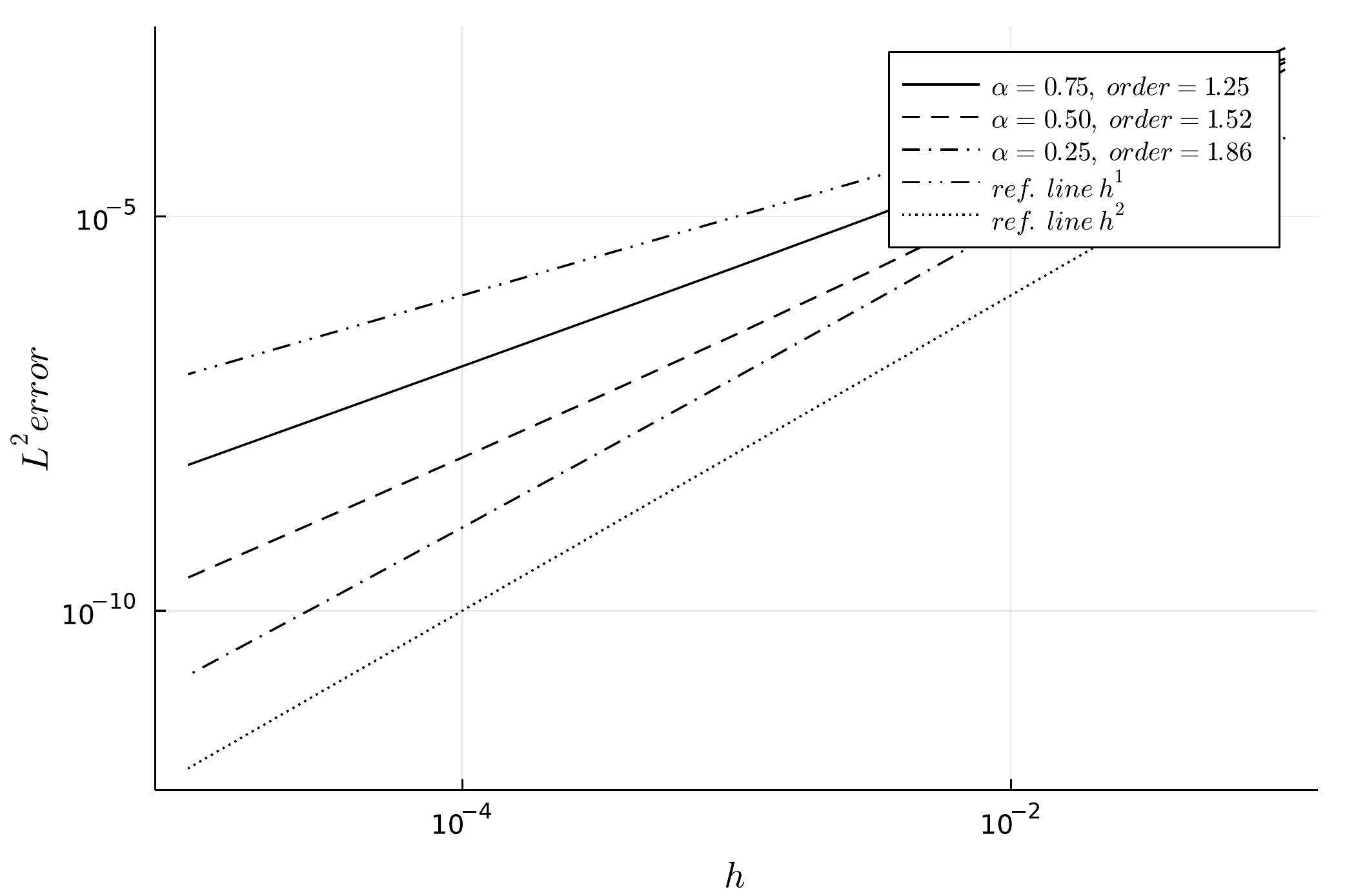}
\caption{A log-log plot of the $L^2$ error for (\ref{eqn:Exact}) and initial condition (\ref{eqn:TIC}) with respect to the time step $h$. Here, $N=5$, $\mu = \alpha$ (top), and $\mu = 2$ (bottom). Calculated order of convergence for different $\alpha$ is given in the legend.}
\label{fig:tOrder}
\end{figure}

\begin{figure}
\centering
\includegraphics[scale = 0.7]{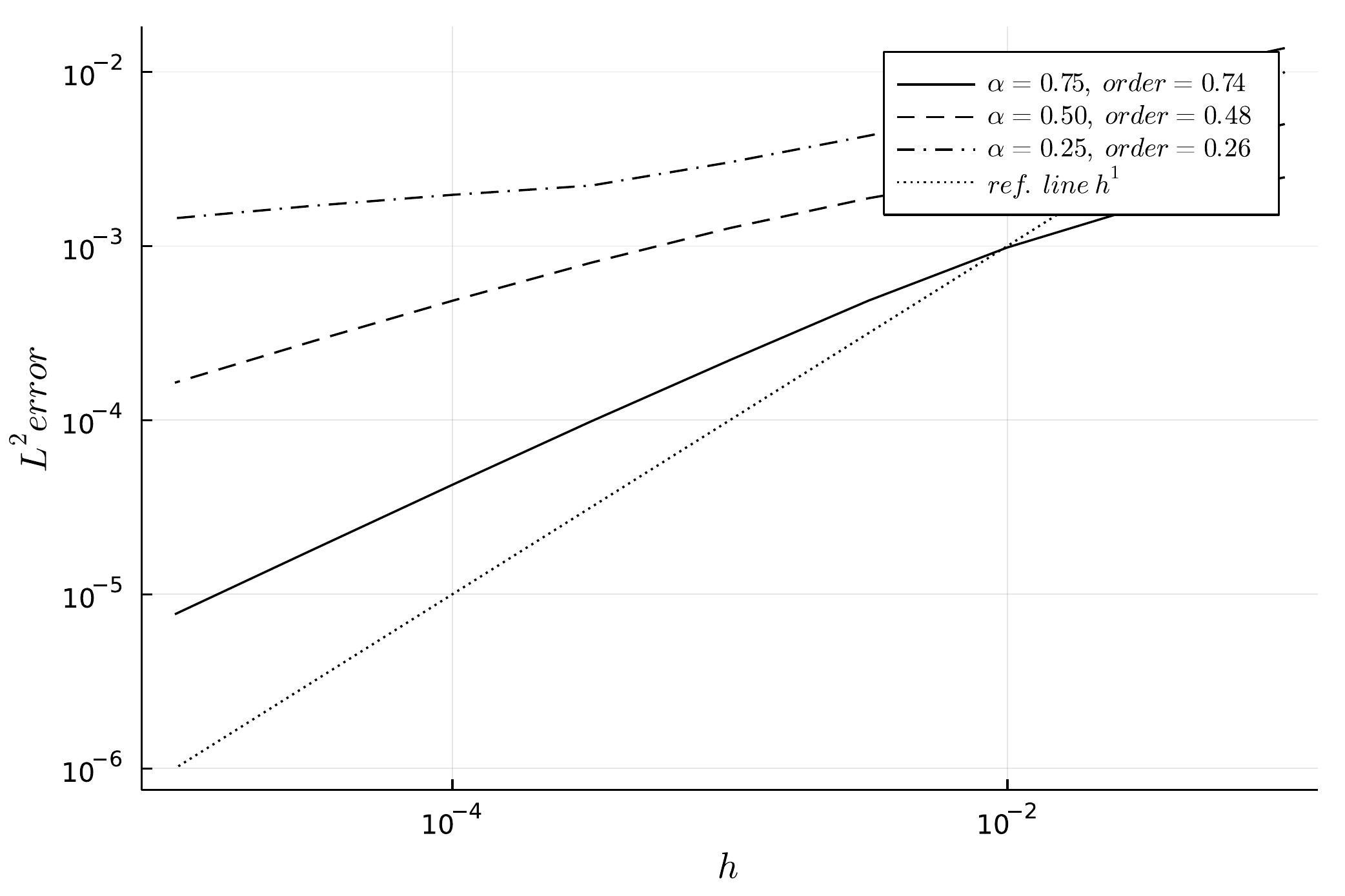}
\caption{A log-log plot of the $L^2$ maximal in time error for (\ref{eqn:Exact}) and initial condition (\ref{eqn:TIC}) with respect to the time step $h$. Here, $N=5$, $\mu = \alpha$. Calculated order of convergence for different $\alpha$ is given in the legend.}
\label{fig:tOrderMax}
\end{figure}

The estimated order for the more realistic example, that is, the Fisher-Kolmogorov equation (\ref{eqn:FK}) is presented in Tab. \ref{tab:FKOrder}. We see that in every considered case, the data are consistent with the typical behavior of the solutions to the subdiffusion equation when the error is calculated at a fixed point. That is, the order is equal to 1. We can also observe that the order estimated from Aitken's method for the maximal error, that is, order$_\infty$, is close to the anticipated $\alpha$. This is true especially for larger values of $\alpha$, say $\alpha\geq 0.5$. On the other hand, even for a very fine grid with spacing $h=2^{-13}$ the calculated order diverges from $\alpha$ for its small values. This is due to extremely slow convergence of the temporal scheme and, above all, approximate nature of the Aitken's extrapolation method. Therefore, we did not obtain a meaningful estimation of the temporal order for a very small $\alpha$. However, the computations shown in Fig. \ref{fig:tOrder}, where we have used exact solutions, verify the claimed order of convergence. Therefore, we can conclude that the numerical computations verify the claim of Theorem \ref{thm:Convergence} and indicate that more analytical work is required to actually show the optimal convergence error stated in Remark \ref{rem:Error}.

\begin{table}
	\centering
	\begin{tabular}{cccccccccccc}
		\toprule
		$\alpha$ & $0.01$ & $0.1$ & $0.2$ & $0.3$ & $0.4$ & $0.5$ & $0.6$ & $0.7$ & $0.8$ & $0.9$ & $0.99$ \\
		\midrule
		order$_\infty$ & - & - & - & $0.17$ & $0.27$ & $0.53$ & $0.58$ & $0.68$ & $0.79$ & $0.90$ & $0.98$\\
		order & $0.99$ & $0.97$ & $0.97$ & $0.98$ & $1.00$ & $1.02$ & $1.03$ & $1.03$ & $1.03$ & $1.03$ & $1.01$\\
		\bottomrule
	\end{tabular}
	\caption{Estimated temporal order of convergence for the Fisher-Kolmogorov equation (\ref{eqn:FK}) for different $\alpha$ using Aitken's method (\ref{eqn:Aitken}) with base $h=2^{-13}$. }
	\label{tab:FKOrder}
\end{table}

As the next test of our scheme, we would like to investigate how well it can resolve the long-time behavior of solutions. To this end, we solve the subdiffusion equation with the following source
\begin{equation}
	\label{eqn:LongTimeSource}
	f(x, t, u) = \pi^2(2u^2+u-1),
\end{equation} 
along with the same diffusivity as before, that is, $D(u) = 1 + u$ and the initial condition $\varphi(x) = x(1-x)$. The exact solution is not available in closed form; however, for long times it will converge to the steady state $v = v(x)$ satisfying
\begin{equation}
	\label{eqn:LongTimeSteady}
	-\left((1+v)v'\right)' = \pi^2 (2v^2 + v -1), \quad v(0) = v(1) = 0.
\end{equation}
It is easy to verify that the solution of the above is $v(x) = \sin(\pi x)$. We would like to investigate numerically the difference $|v(x)-U^n(x)|$ as $t_n\rightarrow\infty$. The results of computations for a representative parameters $\alpha = 0.5$, $N = 10$ at $x = 0.5$ are presented in Fig. \ref{fig:LongTime}. As we can see, the method remains stable as the difference vanishes to zero. Using a finer mesh leads to a more uniform approximation of the valid solution throughout the interval $[0,2^{12}]$ and we obtain decent results for a very modest choice of the grid spacing $h=0.5$. Observe that for all choices of $h$ each curve becomes parallel to $t^{-\alpha}$ at large times. Therefore, we can conclude that our scheme is valid for long-time computations. 

\begin{figure}
	\centering
	\includegraphics[scale = 0.7]{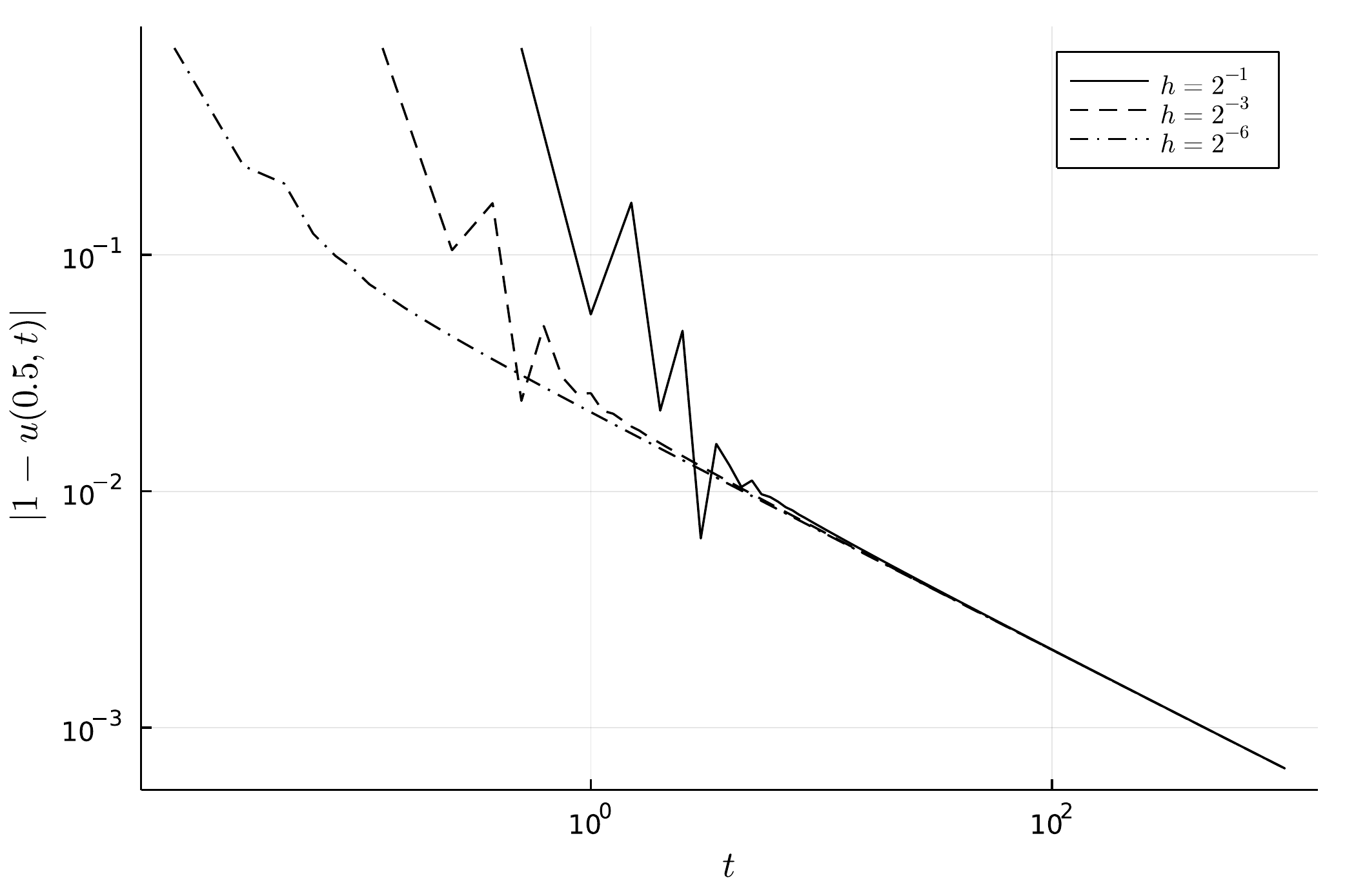}
	\caption{A log-log plot of the difference $|v(0.5)-U^n(0.5)|$ where $U^n$ is the numerical solution of the subdiffusion equation with the source (\ref{eqn:LongTimeSource}) and $v(x) = \sin(\pi x)$. Here, $\alpha = 0.5$, $N = 10$, and $T= 2\times 10^3$. }
	\label{fig:LongTime}
\end{figure}

At the end of this section we would like to investigate the benefit of using a multiple thread parallel computation of the spatial part of the scheme. That is, we compare the times required to evaluate the solution at time $t = 1$ with a fixed time step $h$ for different number of degrees of freedom with and without multi-threaded calculations. As our hardware, we have used a desktop with Intel$\textsuperscript{\textregistered}$ Core$\textsuperscript{\texttrademark}$ i9-12900K processor with 16 cores and 24 threads in order to adequately observe the influence of multithreading. Results are presented in Tab. \ref{tab:MultiThread} where we give the ratio $\tau$ of computation times required to solve the subdiffusion equation with (\ref{eqn:FK}) and $\varphi(x) = \sin (\pi x)$ for multi- and single-threaded evaluation, that is,
\begin{equation}
	\label{eqn:TimeRatio}
	\tau := \frac{\text{single-thread computation time}}{\text{multi-thread computation time}}.
\end{equation}
The computation times have been calculated carefully using the \emph{BenchmarkTools} package in Julia. In order to exclude compilation times from the overall time amount, this package uses statistical sampling for scheme evaluation. As a representative time, we have chosen the median. In these calculations, we vary the number of degrees of freedom $N$ while fixing all other parameters. As we can see, the time required to obtain the numerical solution is an increasing function of degrees of freedom. It can easily be shown that the increase is linear. The benefit of using multithreaded implementation is evident yielding around two times faster computations. 

\begin{table}
	\centering
	\begin{tabular}{cccccccccc}
		\toprule
		$N$ & 7 & 9 & 11 & 13 & 15 & 17 & 19 \\ 
		\midrule
		$\tau$ & 1.34 & 1.77 & 1.59 & 1.86 & 2.20 & 2.45 & 2.56 \\
		\bottomrule
	\end{tabular}
	\caption{Time ratio (\ref{eqn:TimeRatio}) of the single- and multi threaded computation of the solution to the scheme (\ref{eqn:Galerkin}) with (\ref{eqn:FK}) and $\alpha = 0.5$ for a fixed time step $h = 2^{-7}$ and varying degrees of freedom $N$. }
	\label{tab:MultiThread}
\end{table}

\section{Conclusion and future work}
The L1 scheme along with extrapolated in time spectral Galerkin method provides a moderately expensive scheme to solve quasilinear subdiffusion equations. Numerical simulations confirmed the anticipated orders, and our observations indicate that the time of computation is not high if we implement multi-threading. Our future work will be about adapting this technique to the Caputo derivative, i.e., we will investigate the parallel in time integration \cite{Gan07,Xu15}. This will provide even more optimization for finding solutions to equations that are nonlocal in time. Moreover, we are going to rigorously investigate the behavior of the scheme for solutions that are nonsmooth in time, that is, to prove the claim of Remark \ref{rem:Error} for the quasilinear case for the uniform and graded meshes.

\section*{Acknowledgement}
Ł.P. has been supported by the National Science Centre, Poland (NCN) under the grant Sonata Bis with a number NCN 2020/38/E/ST1/00153.

\bibliography{biblio}

\begin{thebibliography}{10}

\bibitem{al2019numerical}
Mariam Al-Maskari and Samir Karaa.
\newblock Numerical approximation of semilinear subdiffusion equations with
  nonsmooth initial data.
\newblock {\em SIAM Journal on Numerical Analysis}, 57(3):1524--1544, 2019.

\bibitem{All16}
Mark Allen, Luis Caffarelli, and Alexis Vasseur.
\newblock A parabolic problem with a fractional time derivative.
\newblock {\em Archive for Rational Mechanics and Analysis}, 221(2):603--630,
  2016.

\bibitem{Amb96}
Fran{\c{c}}ois Amblard, Anthony~C Maggs, Bernard Yurke, Andrew~N Pargellis, and
  Stanislas Leibler.
\newblock Subdiffusion and anomalous local viscoelasticity in actin networks.
\newblock {\em Physical review letters}, 77(21):4470, 1996.

\bibitem{Can07}
Claudio Canuto, M~Yousuff Hussaini, Alfio Quarteroni, and Thomas~A Zang.
\newblock {\em Spectral methods: fundamentals in single domains}.
\newblock Springer Science \& Business Media, 2007.

\bibitem{cen2022time}
Dakang Cen and Zhibo Wang.
\newblock Time two-grid technique combined with temporal second order
  difference method for two-dimensional semilinear fractional sub-diffusion
  equations.
\newblock {\em Applied Mathematics Letters}, 129:107919, 2022.

\bibitem{cen2021second}
Dakang Cen, Zhibo Wang, and Yan Mo.
\newblock Second order difference schemes for time-fractional {KdV--Burgers’}
  equation with initial singularity.
\newblock {\em Applied Mathematics Letters}, 112:106829, 2021.

\bibitem{Cue06}
Eduardo Cuesta, Christian Lubich, and Cesar Palencia.
\newblock Convolution quadrature time discretization of fractional
  diffusion-wave equations.
\newblock {\em Mathematics of Computation}, 75(254):673--696, 2006.

\bibitem{Cui09}
Mingrong Cui.
\newblock Compact finite difference method for the fractional diffusion
  equation.
\newblock {\em Journal of Computational Physics}, 228(20):7792--7804, 2009.

\bibitem{Del05}
Diego del Castillo-Negrete, BA~Carreras, and VE~Lynch.
\newblock Nondiffusive transport in plasma turbulence: a fractional diffusion
  approach.
\newblock {\em Physical {R}eview {L}etters}, 94(6):065003, 2005.

\bibitem{Dip21}
Serena Dipierro, Benedetta Pellacci, Enrico Valdinoci, and Gianmaria Verzini.
\newblock Time-fractional equations with reaction terms: Fundamental solutions
  and asymptotics.
\newblock {\em Discrete \& Continuous Dynamical Systems}, 41(1):257, 2021.

\bibitem{Dip19}
Serena Dipierro, Enrico Valdinoci, and Vincenzo Vespri.
\newblock Decay estimates for evolutionary equations with fractional
  time-diffusion.
\newblock {\em Journal of Evolution Equations}, 19(2):435--462, 2019.

\bibitem{Edw74}
Harold~M Edwards.
\newblock {\em Riemann's {Z}eta function}.
\newblock Academic press, 1974.

\bibitem{El20}
A~El~Abd, SE~Kichanov, M~Taman, KM~Nazarov, DP~Kozlenko, and Wael~M Badawy.
\newblock Determination of moisture distributions in porous building bricks by
  neutron radiography.
\newblock {\em Applied Radiation and Isotopes}, 156:108970, 2020.

\bibitem{El04}
Abd El-Ghany El~Abd and Jacek~J Milczarek.
\newblock Neutron radiography study of water absorption in porous building
  materials: anomalous diffusion analysis.
\newblock {\em Journal of Physics D: Applied Physics}, 37(16):2305, 2004.

\bibitem{Gan07}
Martin~J Gander and Stefan Vandewalle.
\newblock Analysis of the parareal time-parallel time-integration method.
\newblock {\em SIAM Journal on Scientific Computing}, 29(2):556--578, 2007.

\bibitem{Gor15}
Rudolf Gorenflo, Yuri Luchko, and Masahiro Yamamoto.
\newblock Time-fractional diffusion equation in the fractional {Sobolev}
  spaces.
\newblock {\em Fractional Calculus and Applied Analysis}, 18(3):799--820, 2015.

\bibitem{Knu89}
Ronald~L Graham, Donald~E Knuth, Oren Patashnik, and Stanley Liu.
\newblock Concrete mathematics: a foundation for computer science.
\newblock {\em Computers in Physics}, 3(5):106--107, 1989.

\bibitem{Hen02}
Bruce~Ian Henry and Susan~L Wearne.
\newblock Existence of turing instabilities in a two-species fractional
  reaction-diffusion system.
\newblock {\em SIAM Journal on Applied Mathematics}, 62(3):870--887, 2002.

\bibitem{Jin13}
Bangti Jin, Raytcho Lazarov, and Zhi Zhou.
\newblock Error estimates for a semidiscrete finite element method for
  fractional order parabolic equations.
\newblock {\em SIAM Journal on Numerical Analysis}, 51(1):445--466, 2013.

\bibitem{Jin16}
Bangti Jin, Raytcho Lazarov, and Zhi Zhou.
\newblock Two fully discrete schemes for fractional diffusion and
  diffusion-wave equations with nonsmooth data.
\newblock {\em SIAM Journal on {S}cientific {C}omputing}, 38(1):A146--A170,
  2016.

\bibitem{Jin19}
Bangti Jin, Raytcho Lazarov, and Zhi Zhou.
\newblock Numerical methods for time-fractional evolution equations with
  nonsmooth data: a concise overview.
\newblock {\em Computer Methods in Applied Mechanics and Engineering},
  346:332--358, 2019.

\bibitem{Jin18}
Bangti Jin, Buyang Li, and Zhi Zhou.
\newblock An analysis of the {C}rank--{N}icolson method for subdiffusion.
\newblock {\em IMA Journal of Numerical Analysis}, 38(1):518--541, 2018.

\bibitem{Jin19a}
Bangti Jin, Buyang Li, and Zhi Zhou.
\newblock Subdiffusion with a time-dependent coefficient: analysis and
  numerical solution.
\newblock {\em Mathematics of Computation}, 88(319):2157--2186, 2019.

\bibitem{Kla12}
Joseph Klafter, SC~Lim, and Ralf Metzler.
\newblock {\em Fractional dynamics: recent advances}.
\newblock World Scientific, 2012.

\bibitem{Kop19}
Natalia Kopteva.
\newblock Error analysis of the {L1} method on graded and uniform meshes for a
  fractional-derivative problem in two and three dimensions.
\newblock {\em Mathematics of Computation}, 88(319):2135--2155, 2019.

\bibitem{Kou08}
Samuel~C Kou.
\newblock Stochastic modeling in nanoscale biophysics: subdiffusion within
  proteins.
\newblock {\em The Annals of Applied Statistics}, 2(2):501--535, 2008.

\bibitem{Lan05}
TAM Langlands and Bruce~I Henry.
\newblock The accuracy and stability of an implicit solution method for the
  fractional diffusion equation.
\newblock {\em Journal of Computational Physics}, 205(2):719--736, 2005.

\bibitem{Li19}
Changpin Li and Min Cai.
\newblock {\em Theory and numerical approximations of fractional integrals and
  derivatives}.
\newblock SIAM, 2019.

\bibitem{Li14}
Changpin Li and Hengfei Ding.
\newblock Higher order finite difference method for the reaction and
  anomalous-diffusion equation.
\newblock {\em Applied Mathematical Modelling}, 38(15-16):3802--3821, 2014.

\bibitem{Li19a}
Changpin Li and Fanhai Zeng.
\newblock {\em Numerical methods for fractional calculus}.
\newblock Chapman and Hall/CRC, 2019.

\bibitem{Li18}
Dongfang Li, Hong-Lin Liao, Weiwei Sun, Jilu Wang, and Jiwei Zhang.
\newblock Analysis of {L}1-{G}alerkin {FEM}s for time-fractional nonlinear
  parabolic problems.
\newblock {\em Communications in Computational Physics}, 24(1):86--103, 2018.

\bibitem{Li19c}
Dongfang Li, Chengda Wu, and Zhimin Zhang.
\newblock Linearized {G}alerkin {FEM}s for nonlinear time fractional parabolic
  problems with non-smooth solutions in time direction.
\newblock {\em Journal of {S}cientific {C}omputing}, 80(1):403--419, 2019.

\bibitem{Li10}
Wulan Li and Xu~Da.
\newblock Finite central difference/finite element approximations for parabolic
  integro-differential equations.
\newblock {\em Computing}, 90(3-4):89--111, 2010.

\bibitem{Li09}
Xianjuan Li and Chuanju Xu.
\newblock A space-time spectral method for the time fractional diffusion
  equation.
\newblock {\em SIAM Journal on Numerical Analysis}, 47(3):2108--2131, 2009.

\bibitem{Lia18}
Hong-lin Liao, Dongfang Li, and Jiwei Zhang.
\newblock Sharp error estimate of the nonuniform {L}1 formula for linear
  reaction-subdiffusion equations.
\newblock {\em SIAM Journal on Numerical Analysis}, 56(2):1112--1133, 2018.

\bibitem{Lia19}
Hong-lin Liao, William McLean, and Jiwei Zhang.
\newblock A discrete {G}r\"onwall inequality with applications to numerical
  schemes for subdiffusion problems.
\newblock {\em SIAM Journal on Numerical Analysis}, 57(1):218--237, 2019.

\bibitem{Lin07}
Yumin Lin and Chuanju Xu.
\newblock Finite difference/spectral approximations for the time-fractional
  diffusion equation.
\newblock {\em Journal of computational physics}, 225(2):1533--1552, 2007.

\bibitem{Liu04}
Fawang Liu, Vo~Anh, and Ian Turner.
\newblock Numerical solution of the space fractional {Fokker--Planck} equation.
\newblock {\em Journal of Computational and Applied Mathematics},
  166(1):209--219, 2004.

\bibitem{Liu07}
Fawang Liu, Pinghui Zhuang, Vo~Anh, Ian Turner, and Kevin Burrage.
\newblock Stability and convergence of the difference methods for the
  space--time fractional advection--diffusion equation.
\newblock {\em Applied Mathematics and Computation}, 191(1):12--20, 2007.

\bibitem{Liu18}
Wei Liu, Michael R\"ockner, and Jos{\'e}~Lu{\'\i}s da~Silva.
\newblock Quasi-linear (stochastic) partial differential equations with
  time-fractional derivatives.
\newblock {\em SIAM Journal on Mathematical Analysis}, 50(3):2588--2607, 2018.

\bibitem{Lub88}
Christian Lubich.
\newblock Convolution quadrature and discretized operational calculus. i.
\newblock {\em Numerische Mathematik}, 52(2):129--145, 1988.

\bibitem{McL10}
William McLean.
\newblock Regularity of solutions to a time-fractional diffusion equation.
\newblock {\em The ANZIAM Journal}, 52(2):123--138, 2010.

\bibitem{Met00}
Ralf Metzler and Joseph Klafter.
\newblock The random walk's guide to anomalous diffusion: a fractional dynamics
  approach.
\newblock {\em Physics reports}, 339(1):1--77, 2000.

\bibitem{Mul96}
H-P M{\"u}ller, Rainer Kimmich, and Jan Weis.
\newblock {NMR} flow velocity mapping in random percolation model objects:
  Evidence for a power-law dependence of the volume-averaged velocity on the
  probe-volume radius.
\newblock {\em Physical Review E}, 54(5):5278, 1996.

\bibitem{Mus18}
Kassem Mustapha.
\newblock {FEM} for time-fractional diffusion equations, novel optimal error
  analyses.
\newblock {\em Mathematics of Computation}, 87(313):2259--2272, 2018.

\bibitem{Mus09}
Kassem Mustapha and William McLean.
\newblock Discontinuous {G}alerkin method for an evolution equation with a
  memory term of positive type.
\newblock {\em Mathematics of computation}, 78(268):1975--1995, 2009.

\bibitem{Okr21}
Hanna Okrasi{\'n}ska-P{\l}ociniczak and {\L}ukasz P{\l}ociniczak.
\newblock Second order scheme for self-similar solutions of a time-fractional
  porous medium equation on the half-line.
\newblock {\em Applied Mathematics and Computation}, 424:127033, 2022.

\bibitem{Old74}
Keith Oldham and Jerome Spanier.
\newblock {\em The fractional calculus theory and applications of
  differentiation and integration to arbitrary order}.
\newblock Elsevier, 1974.

\bibitem{ou2022mathematical}
Caixia Ou, Dakang Cen, Seakweng Vong, and Zhibo Wang.
\newblock Mathematical analysis and numerical methods for {Caputo-Hadamard}
  fractional diffusion-wave equations.
\newblock {\em Applied Numerical Mathematics}, 177:34--57, 2022.

\bibitem{Plo14}
{\L}ukasz P{\l}ociniczak.
\newblock Approximation of the {E}rd\'elyi--{Ko}ber operator with application
  to the time-fractional porous medium equation.
\newblock {\em SIAM Journal on Applied Mathematics}, 74(4):1219--1237, 2014.

\bibitem{Plo15}
{\L}ukasz P{\l}ociniczak.
\newblock Analytical studies of a time-fractional porous medium equation.
  derivation, approximation and applications.
\newblock {\em Communications in Nonlinear Science and Numerical Simulation},
  24(1):169--183, 2015.

\bibitem{Plo19}
{\L}ukasz P{\l}ociniczak.
\newblock Numerical method for the time-fractional porous medium equation.
\newblock {\em SIAM Journal on Numerical Analysis}, 57(2):638--656, 2019.

\bibitem{plociniczak2022error}
{\L}ukasz P{\l}ociniczak.
\newblock Error of the galerkin scheme for a semilinear subdiffusion equation
  with time-dependent coefficients and nonsmooth data.
\newblock {\em arXiv preprint arXiv:2202.13728}, 2022.

\bibitem{Plo21}
{\L}ukasz P{\l}ociniczak.
\newblock Linear galerkin-legendre spectral scheme for a degenerate nonlinear
  and nonlocal parabolic equation arising in climatology.
\newblock {\em Applied Numerical Mathematics}, 179:105--124, 2022.

\bibitem{Plo17}
{\L}ukasz P{\l}ociniczak and Hanna Okrasi{\'n}ska-P{\l}ociniczak.
\newblock Numerical method for {V}olterra equation with a power-type
  nonlinearity.
\newblock {\em Applied Mathematics and Computation}, 337:452--460, 2018.

\bibitem{Plo17b}
{\L}ukasz P{\l}ociniczak and Szymon Sobieszek.
\newblock Numerical schemes for integro-differential equations with
  erd{\'e}lyi-kober fractional operator.
\newblock {\em Numerical Algorithms}, 76(1):125--150, 2017.

\bibitem{Plo17a}
{\L}ukasz P{\l}ociniczak and Mateusz \'Swita{\l}a.
\newblock Existence and uniqueness results for a time-fractional nonlinear
  diffusion equation.
\newblock {\em Journal of Mathematical Analysis and Applications},
  462(2):1425--1434, 2018.

\bibitem{qiao2020alternating}
Leijie Qiao, Zhibo Wang, and Da~Xu.
\newblock An alternating direction implicit orthogonal spline collocation
  method for the two dimensional multi-term time fractional
  integro-differential equation.
\newblock {\em Applied Numerical Mathematics}, 151:199--212, 2020.

\bibitem{Ren20}
Jincheng Ren, Dongyang Shi, and Seakweng Vong.
\newblock High accuracy error estimates of a {G}alerkin finite element method
  for nonlinear time fractional diffusion equation.
\newblock {\em Numerical Methods for Partial Differential Equations},
  36(2):284--301, 2020.

\bibitem{Sak11}
Kenichi Sakamoto and Masahiro Yamamoto.
\newblock Initial value/boundary value problems for fractional diffusion-wave
  equations and applications to some inverse problems.
\newblock {\em Journal of Mathematical Analysis and Applications},
  382(1):426--447, 2011.

\bibitem{Sch06}
Achim Sch{\"a}dle, Mar{\'\i}a L{\'o}pez-Fern{\'a}ndez, and Christian Lubich.
\newblock Fast and oblivious convolution quadrature.
\newblock {\em SIAM Journal on Scientific Computing}, 28(2):421--438, 2006.

\bibitem{She11}
Jie Shen, Tao Tang, and Li-Lian Wang.
\newblock {\em Spectral methods: algorithms, analysis and applications},
  volume~41.
\newblock Springer Science \& Business Media, 2011.

\bibitem{Sty16}
Martin Stynes.
\newblock Too much regularity may force too much uniqueness.
\newblock {\em Fractional Calculus and Applied Analysis}, 19(6):1554--1562,
  2016.

\bibitem{stynes2017error}
Martin Stynes, Eugene O'Riordan, and Jos{\'e}~Luis Gracia.
\newblock Error analysis of a finite difference method on graded meshes for a
  time-fractional diffusion equation.
\newblock {\em SIAM Journal on Numerical Analysis}, 55(2):1057--1079, 2017.

\bibitem{Sun17}
Titiwat Sungkaworn, Marie-Lise Jobin, Krzysztof Burnecki, Aleksander Weron,
  Martin~J Lohse, and Davide Calebiro.
\newblock Single-molecule imaging reveals receptor--g protein interactions at
  cell surface hot spots.
\newblock {\em Nature}, 550(7677):543, 2017.

\bibitem{Tho07}
Vidar Thom{\'e}e.
\newblock {\em Galerkin finite element methods for parabolic problems},
  volume~25.
\newblock Springer Science \& Business Media, 2007.

\bibitem{Top17}
Erwin Topp and Miguel Yangari.
\newblock Existence and uniqueness for parabolic problems with {Caputo} time
  derivative.
\newblock {\em Journal of Differential Equations}, 262(12):6018--6046, 2017.

\bibitem{Vaz17}
Juan~Luis V{\'a}zquez.
\newblock The mathematical theories of diffusion: Nonlinear and fractional
  diffusion.
\newblock In {\em Nonlocal and Nonlinear Diffusions and Interactions: New
  Methods and Directions}, pages 205--278. Springer, 2017.

\bibitem{Ver15}
Vicente Vergara and Rico Zacher.
\newblock Optimal decay estimates for time-fractional and other nonlocal
  subdiffusion equations via energy methods.
\newblock {\em SIAM Journal on Mathematical Analysis}, 47(1):210--239, 2015.

\bibitem{wang2022second}
Zhibo Wang, Caixia Ou, and Seakweng Vong.
\newblock A second-order scheme with nonuniform time grids for
  {Caputo-Hadamard} fractional sub-diffusion equations.
\newblock {\em Journal of Computational and Applied Mathematics}, page 114448,
  2022.

\bibitem{Wit21}
Petra Wittbold, Patryk Wolejko, and Rico Zacher.
\newblock Bounded weak solutions of time-fractional porous medium type and more
  general nonlinear and degenerate evolutionary integro-differential equations.
\newblock {\em Journal of Mathematical Analysis and Applications},
  499(1):125007, 2021.

\bibitem{Xu15}
Qinwu Xu, Jan~S Hesthaven, and Feng Chen.
\newblock A parareal method for time-fractional differential equations.
\newblock {\em Journal of Computational Physics}, 293:173--183, 2015.

\bibitem{Yus05}
Santos~B Yuste and Luis Acedo.
\newblock An explicit finite difference method and a new von {Neumann-type}
  stability analysis for fractional diffusion equations.
\newblock {\em SIAM Journal on Numerical Analysis}, 42(5):1862--1874, 2005.

\bibitem{Zac12}
Rico Zacher.
\newblock Global strong solvability of a quasilinear subdiffusion problem.
\newblock {\em Journal of Evolution Equations}, 12(4):813--831, 2012.

\bibitem{Zak20}
Mahmoud~A Zaky, Ahmed~S Hendy, and Jorge~E Mac{\'\i}as-D{\'\i}az.
\newblock Semi-implicit {G}alerkin--{L}egendre spectral schemes for nonlinear
  time-space fractional diffusion--reaction equations with smooth and nonsmooth
  solutions.
\newblock {\em Journal of {S}cientific {C}omputing}, 82(1):1--27, 2020.

\bibitem{Zen15}
Fanhai Zeng, Changpin Li, Fawang Liu, and Ian Turner.
\newblock Numerical algorithms for time-fractional subdiffusion equation with
  second-order accuracy.
\newblock {\em SIAM Journal on Scientific Computing}, 37(1):A55--A78, 2015.

\bibitem{Zha19}
Kangqun Zhang.
\newblock Existence results for a generalization of the time-fractional
  diffusion equation with variable coefficients.
\newblock {\em Boundary Value Problems}, 2019(1):1--11, 2019.

\bibitem{Zhu06}
Pinghui Zhuang and Fawang Liu.
\newblock Implicit difference approximation for the time fractional diffusion
  equation.
\newblock {\em Journal of Applied Mathematics and Computing}, 22(3):87--99,
  2006.

\bibitem{Zhu16}
Pinghui Zhuang, Fawang Liu, Ian Turner, and Vo~Anh.
\newblock Galerkin finite element method and error analysis for the fractional
  cable equation.
\newblock {\em Numerical Algorithms}, 72(2):447--466, 2016.

\end{thebibliography}
\bibliographystyle{plain}

\end{document}